\documentclass{article}

\usepackage{amssymb}
\usepackage{amsmath}
\usepackage{amsthm}
\usepackage{mathrsfs}
\usepackage{enumerate}
\usepackage{xcolor}
\usepackage[
backend=biber,
style=alphabetic,
]{biblatex}
\usepackage[colorlinks=true,citecolor=red,linkcolor=blue]{hyperref}

\newtheorem{mainthm}{Main Theorem}
\newtheorem{theorem}{Theorem}[section]
\newtheorem{lemma}[theorem]{Lemma}
\newtheorem{proposition}[theorem]{Proposition}
\newtheorem{corollary}[theorem]{Corollary}
\newtheorem*{definition}{Definition}
\newtheorem*{remark}{Remark}

\newcommand{\norm}[3]{{\mathbf{N}}_{{#1}/{#2}}({#3})}
\newcommand{\congruente}[3]{#1 \equiv #2 \pmod{#3}}
\newcommand{\ncongruente}[3]{#1 \not\equiv #2 \pmod{#3}}
\newcommand{\legendre}[2]{\ensuremath{\left(\frac{#1}{#2}\right)}}
\providecommand{\keywords}[1]
{
	\small	
	\textbf{\textit{Keywords:}} #1
}
\providecommand{\subjclass}[1]
{
	\small	
	\textbf{\textit{MSC2020:}} #1
}
\title{Undecidability of infinite algebraic extensions of $\mathbb{F}_p(t)$}
\author{
    Carlos Martínez-Ranero\\
    Dubraska Salcedo\\
    Javier Utreras\\
    {\small Departamento de Matemática}\\
    {\small Universidad de Concepción}
}
\date{}

\begin{document}

\maketitle

\begin{abstract}
    Building on work of J. Robinson and A. Shlapentokh, we develop a general framework to obtain definability and decidability results of large classes of infinite algebraic extensions of $\mathbb{F}_p(t)$. As an application, we show that for every odd rational prime $p$ there exist infinitely many primes $r$ such that the fields  $\mathbb{F}_{p^a}\left(t^{r^{-\infty}}\right)$ have undecidable first-order theory in the language of rings without parameters. Our method uses character theory to construct families of non-isotrivial elliptic curves whose Mordell-Weil group is finitely generated  and of positive rank in $\mathbb{Z}_r$-towers.\footnote{ The first named author was partially supported by  Proyecto VRID-Investigación  No. 220.015.024-INV\\ The second named author was supported by Beca ANID Doctorado Nacional 21200661 \\The third named author was supported by Proyecto VRID 2022000419-INI}
\end{abstract}


\section{Introduction}\label{sec1}

Though the search for systematic procedures or algorithms to solve special type of equations has been considered since antiquity, it became one of the central topics of research in the 20th century  driven by some famous questions posed by D. Hilbert (e.g. Hilbert's 10th problem and Hilbert’s Entscheidungsproblem)and by the pioneering work of K. G\"odel, A. Church and A. Turing in the theoretical development of the notion of computability.

Notably, B. Rosser \cite{rosser}, building on the celebrated result of K. G\"odel \cite{godel}, showed that the first-order theory of $\mathbb{Z}$ is undecidable, i.e. there is no algorithm (Turing machine) which can answer every \emph{arithmetic} (yes/no)-question correctly. This undecidability result was extended to other number theoretic objects like rings of integers and number fields by J. Robinson \cite{robinson59}. On the other hand, at around the same time A. Tarski \cite{tarski2} showed that the rings $\mathbb{Q}^{\rm alg}$ and $\mathbb{Q}^{\rm alg}\cap \mathbb{R}$ are decidable. This was extended later on by J. Ax and S. Kochen \cite{axkochen} to the field $\mathbb{Q}^{\rm alg}\cap \mathbb{Q}_p$ of $p$-adic algebraic numbers. Thus the following question, first considered by J. Robinson in \cite{robinson49}, naturally arises: 

\begin{quote}
    \textbf{Problem.} Which (possibly infinite) algebraic extensions of $\mathbb{Q}$ have an undecidable first-order theory?
\end{quote}

The question for infinite algebraic extensions is more subtle as from the aforementioned results of A. Tarski it follows that there exist extensions whose first-order theory is decidable. Intuitively, fields close to ${\mathbb{Q}^{\rm alg}}$ are expected to be decidable, whereas fields closer to $\mathbb{Q}$ are expected to be undecidable.

J. Robinson showed in \cite{robinson59} that the first-order theory of the ring of all totally real algebraic integers is undecidable and produced a "blueprint" for rings of integers which are not necessarily totally real. Later on, building on work of R. Rumely \cite{rumely}, C. Videla \cite{videlaprop} give a definition of the ring of integers on pro-$p$ Galois extensions. This was extended by K. Fukuzaki \cite{fukuzaki} and A. Shlapentokh \cite{sasha}. By combining J. Robinson's blueprint with the definability results of C. Videla, K. Fukuzaki and A. Shlapentokh, several authors have produced families of infinite algebraic extensions of $\mathbb{Q}$ with undecidable first order theory (see \cite{videlaelliptic},\cite{vv}, \cite{MUV},\cite{caleb}).

The aim of this paper is to extend the aforementioned results to the context of infinite algebraic extensions of $\mathbb{F}_p(t)$. It is worth pointing out that the undecidability of the first-order theory of rings of positive characteristic has been a very active area of research, mostly linked with the analogue of Hilbert's tenth problem for function fields. Hilbert’s Tenth Problem is known to be undecidable for global function fields (see \cite{pheidas}, \cite{videlac2}, \cite{sashaff}, \cite{eisentrager}), and the first-order theory of any function field of positive characteristic in the language of rings without parameters is undecidable \cite{kirstensasha}. The undecidability problem for infinite algebraic extensions of $\mathbb{F}_p(t)$, to the best of the authors' knowledge, has not being touched.

To begin, we build upon the foundational work of J. Robinson and C. W. Henson (unpublished, refer to \cite{vandendries}) to establish a highly versatile framework for proving the undecidability of a wide range of commutative rings that encompass meaningful arithmetic. Notably, this framework applies seamlessly to rings of algebraic integers and, in a broader context, to rings of $S$-integers, serving as illustrative prototypes.

\begin{mainthm}[Theorem \ref{thm:henson}] \label{mainthm:henson}  
    Let $S$ be an infinite commutative domain satisfying the following properties:
    \begin{itemize}
        \item For any finite set $A\subseteq S$ there exists an element $x$, neither zero nor a unit, such that $(x)+(a_i)=S$ for every $a_i\in A\setminus\left\{0\right\}$.
        \item For any finite set $A\subseteq S$ that does not contain zero and any $a\in S$ that is neither a zero nor a unit there exists an element $g\in S$ such that
        \begin{itemize}
            \item for any $a_i\in A$, $1+a_ig$ is neither zero nor a unit;
            \item for any $a_i\in A$, $(a)+(1+a_ig) = S$; and
            \item for distinct $a_i,a_j\in A$, $(1+a_ig) + (1+a_jg) = S$.
        \end{itemize} 
    \end{itemize}
    Consider a first--order structure on $S$ (which we will also denote with $S$) extending the ring structure.
    If there exists a uniformly parametrizable family of subsets of $S$ in this structure such that it contains sets of arbitrarily large finite cardinalities, then there exists a model of the theory $R$ of Tarski, Mostovski and Robinson interpretable in $S$.
    Hence the theory of $S$ is undecidable.
\end{mainthm}

Next we extend the results of A. Shlapentokh to the positive characteristic and give a general criterion to define the ring of integers of an algebraic extension of a function field.

\begin{mainthm}[Theorem \ref{thm:sasha}] \label{mainthm:sasha}
    Let $K$ be a finite separable geometric field extension of $\mathbb{F}_{p^a}(t), $ and let ${\mathscr S}_{K}$ be a finite nonempty set of primes of $K$. Let $\ell, \ell'$ be rational primes (not necessarily different) and coprime to the characteristic of $\mathbb{F}_{p^a}.$  Let $K_{\rm inf}$ be an infinite algebraic extension of $K$. Suppose all primes of $K$ are hereditarily $\ell$-bounded and all primes in ${\mathscr S}_K$  are completely $\ell'$-bounded, then the integral closure $\mathcal{O}_{K_{\rm inf},\mathscr{S}_{{\rm inf}}}$ of $\mathcal{O}_{K,{\mathscr S}_K}$ in $K_{\rm inf}$ is first-order definable with parameters. 
\end{mainthm}

The proof is a straightforward adaptation of the methods of \cite{sasha} to the function field case. However, in order to not place the burden of proof on the reader, we shall provide a streamlined proof. 

Finally, as an application we show the following.

\begin{mainthm}[Theorem \ref{thm:main}] \label{mainthm:main}
    Let $p$ be an odd prime, $a$ a positive integer and $r$ a prime congruent to $3$ $\pmod{4}$ and such that $p$ is a quadratic residue modulo $r$.

    The first-order theory of the structure $\mathbb{F}_{p^a}\left(t^{r^{-\infty}}\right)$ in the language of rings $\{0,1,+,\times\}$ interprets the theory R of Tarski, Mostowski and Robinson. In particular, it is undecidable.
\end{mainthm}

Our method uses character theory and results by R. Conceição, C. Hall and D. Ulmer \cite{ulmerleg2} to construct families of non-isotrivial elliptic curves whose Mordell-Weil
group is finitely generated and of positive rank in $\mathbb{Z}_r$-towers.

The structure of the paper is as follows. Section \ref{sec:henson} details the proof of Main Theorem \ref{mainthm:henson}. Main Theorem \ref{mainthm:sasha} is covered in Section \ref{sec:sasha}. In Section~\ref{sec:ulmer} we construct a family of definable sets that satisfies the criterion given by Main Theorem~\ref{mainthm:henson} using results on a specific elliptic curve over a function field by R. Conceição, C. Hall and D. Ulmer \cite{ulmerleg2}. Finally, Main Theorem \ref{mainthm:main} is proved in Section~\ref{sec:mainthm}.

It is worth mentioning that Sections~\ref{sec:henson}, \ref{sec:sasha} and \ref{sec:ulmer} can be read independently. Furthermore, if one is willing to take the main theorems from each of these sections for granted, then Section~\ref{sec:mainthm}, which consists of the proof of Main Theorem~\ref{mainthm:main}, can also be read on its own.

\section{A criterion for undecidability of rings} \label{sec:henson}

The aim of this section is to prove the first Main Theorem. In order to do so, we will seek to interpret a known essentially undecidable theory.

\subsection{The theory $R$ of Tarski, Mostwoski and Robinson}

Let $\mathcal{L}_{\rm arit}=\{0,S,+,\cdot\}$ be the language of arithmetic. Define recursively, for each $n$, $\Delta_0=0$ and $\Delta_{n+1}=S(\Delta_n)$. Write $\alpha \leq \beta$ instead of $\exists w (w+\alpha=\beta)$, where $\alpha, \beta$ are arbitrary terms and $w$ is any variable which occurs in neither $\alpha$ nor $\beta$.

In \cite{tarski} A. Tarski, A. Mostwoski and R. Robinson consider the theory $R$ given by the following axiom instances for any natural integers $n, m$:

\begin{enumerate}
    \item[$\Omega_1.$] $\Delta_n+\Delta_m=\Delta_{n+m};$
    \item[$\Omega_2.$] $\Delta_n \cdot \Delta_m=\Delta_{nm};$
    \item[$\Omega_3.$] $\Delta_n \ne \Delta_m$ for $n\ne m;$
    \item[$\Omega_4.$] $\forall x\,\left(x\leq \Delta_n \implies x=\Delta_0\vee\dots\vee x=\Delta_n\right)$;
    \item[$\Omega_5.$] $\forall x\,\left(x\leq \Delta_n \vee \Delta_n\leq x\right)$.
\end{enumerate}

\begin{theorem} [\cite{tarski}, Theorem II.9]
    Any model of $R$ has undecidable theory.
\end{theorem}

\subsection{An undecidability criterion for rings}

\begin{definition}
    Let $S$ be any ring and let $\mathcal{F}\subseteq \mathcal{P}(S)$. We say that $\mathcal{F}$ is a uniformly parametrizable family of subsets of $S$ if and only if there exist a definable set $C\subseteq S^k$ (the parameter set) and a first-order formula $\varphi(x,y_1,\dots,y_k)$ (the formula of definition) such that for any $F\subset S$
    \[
    F\in \mathcal{F} \Longleftrightarrow \exists\, \overline{y}\in C\, \forall x\in S\, \left[x\in F \Longleftrightarrow \varphi(x,\overline{y})\right]. 
    \]
\end{definition}

We show the following general criterion:

\begin{theorem} \label{thm:henson}
    Let $S$ be an infinite commutative domain satisfying the following properties:
    \begin{itemize}
        \item For any finite set $A\subseteq S$ there exists an element $x$, neither zero nor a unit, such that $(x)+(a_i)=S$ for every $a_i\in A\setminus\left\{0\right\}$.
        \item For any finite set $A\subseteq S$ that does not contain zero and any $a\in S$ that is neither a zero nor a unit there exists an element $g\in S$ such that
        \begin{itemize}
            \item for any $a_i\in A$, $1+a_ig$ is neither zero nor a unit;
            \item for any $a_i\in A$, $(a)+(1+a_ig) = S$; and
            \item for distinct $a_i,a_j\in A$, $(1+a_ig) + (1+a_jg) = S$.
        \end{itemize} 
    \end{itemize}

    Consider a first--order structure on $S$ (which we will also denote with $S$) extending the ring structure.

    If there exists a uniformly parametrizable family of subsets of $S$ in this structure such that it contains sets of arbitrarily large finite cardinalities, then there exists a model of the theory $R$ interpretable in $S$.

    Hence the theory of $S$ is undecidable.
\end{theorem}

The proof of this theorem will require a number of definitions. Until the end of this proof, let $S$ be a ring satisfying the conditions of the statement, and let $\mathcal{C}$ be a uniformly parametrizable family of subsets of $S$ that contains sets of arbitrarily large finite cardinalities.

\subsubsection{Sketch of the proof}

As the proof might get quite dense, we will provide an overview of the steps we will later follow: we will aim to interpret a model of the theory $R$ as a quotient of our parametrizable family; a first approach would be to make it so that each class consists of sets of a fixed cardinality. This, of course, may be impossible, as ``existence of a bijection between two given definable sets'' does not seem to be first--order definable. However, we will ensure that it behaves nicely for finite sets by carrying out the following steps:

\begin{itemize}
    \item Add sets to the family $\mathcal{C}$ so that it contains sets of every finite cardinality and does not contain $S$; this generates a new parametrized family $\mathcal{C}_{0}$. (cf. Lemma~\ref{def:cfill} and  Definition~\ref{def:c0})
    \item Define a binary relation $\psi$ between sets of $\mathcal{C}_{0}$. This relation will behave like ``less or equal cardinality than'' when involving a finite set. (cf. Lemma~\ref{lem:definject})
    \item Using this relation, define a relation $\sim$ between sets of $\mathcal{C}_{0}$. Between finite sets, this will behave like ``equal cardinality''. (cf. Definition~\ref{def:cardequiv})
    \item Restrict the parameter set of our parametrized family to obtain a smaller family $\mathcal{C}_{\rm equiv}$ where $\sim$ is an equivalence relation. (cf. Definition~\ref{def:Cequiv})
    \item Define a ternary relation between sets of $\mathcal{C}_{\rm equiv}$. Our aim is to have this relation satisfy: $(A,B,C)$ are related iff $|A|+|B|=|C|$. This will not always be the case, but it will be satisfied for finite sets. (cf. Definitions~\ref{def:equivplus} and \ref{def:Cplus})
    \item Do the same for a ternary relation that should behave like multiplication of cardinalities. Once again, this will hold for finite sets (cf. Definition~\ref{def:equivtimes}).
    \item Restrict the parameter set of our parametrized family to obtain a smaller family $\mathcal{C}_{+,\times}$ where the two ternary relations respect $\sim$. (cf. Definition~\ref{def:Cplustimes})
\end{itemize}

At this point, the quotient set of our family modulo $\sim$ will look like a structure with addition and multiplication. There will be two final adjustments to make:

\begin{itemize}
    \item Restrict the parameter set of our parametrized family to obtain a smaller family $\mathcal{C}_{\rm bound}$ where for any $A$ finite and $B$ infinite there exists $C$ infinite such that $C+A=B$ (this is required as an axiom of the theory $R$). (cf. Definition~\ref{def:Cbound} and Proposition~\ref{prop:finite_substraction}) 
    \item By this point, the operations may be only partial. We include a placeholder ``infinite'' set in our family such that any still undefined addition or multiplication returns that set. (cf. Definition~\ref{def:Cbar})
\end{itemize}

\subsubsection{Filling up the family to include all finite cardinalities}

\begin{lemma} \label{def:cfill}
    There exists a uniformly parametrizable family $\mathcal{C}_{\rm fill}$ of subsets of $S$ such that $\mathcal{C}\subseteq \mathcal{C}_{\rm fill}$ and $\mathcal{P}(F)\subseteq\mathcal{C}_{\rm fill}$ for any finite $F\in\mathcal{C}$.
\end{lemma}

\begin{proof}
    Let $\varphi(x,y_1,\dots,y_k)$ be the uniform parametrization of $\mathcal{C}$. Consider the formula $\Phi(x,y_1,\dots,y_{k+3})$ defined by
    \[
    \varphi(x,y_1,\dots,y_k)\, \wedge\, \left(1+xy_{k+1} \mid y_{k+2}\right) \,\wedge \,\left(x\neq y_{k+3}\right).
    \]

    We claim that the family $\mathcal{C}_{\rm fill}$ parametrized by $\Phi$ is as required. Let $b_1,\dots,b_k$ be such that $\varphi(S,b_1,\dots,b_k)=\{a_1,\dots,a_n\}=F$ is finite, and let $G\subset F.$ We need to show that for some parameters $b_{k+1},b_{k+2}$ and $b_{k+3}$ we have that
    \[
    \psi(S,b_1,\dots,b_{k+3})=G.
    \]

    By the hypothesis on $S$, we can choose $g$ such that the ideals $I_i=(1+a_ig)$ are pairwise comaximal and the elements $1+a_ig$ are neither zero nor units for any $i=1,\dots,n$, except if $a_i=0$. Set $b_{k+1}=g, b_{k+2}=\prod\limits_{a_i\in G}(1+a_ib_{k+1})$ and $b_{k+3}=0$ if $0\not\in G$, or $b_{k+3}\not\in G$ otherwise. Note that if $a_j\in F\setminus G$, then $1+a_jb_{k+1}\notin \bigcap\limits_{a_i\in G} I_i=\left(\prod\limits_{a_i\in G}\left(1+a_ib_{k+1}\right)\right)$, thus it follows that
    \[
    x\in G \Longleftrightarrow \varphi(x,b_1,\dots,b_k) \wedge 1+xb_{k+1} \mid b_{k+2} \wedge x\ne b_{k+3}. 
    \]
\end{proof}

\begin{definition} \label{def:c0}
    The uniformly parametrizable family $\mathcal{C}_0$ is defined as $\mathcal{C}_{\rm fill}\setminus\{S\}$.
\end{definition}

\subsubsection{Definable uniform injection}

Extend the language of the structure with two unary predicate symbols $P_1$, $P_2$. For any $A,B \subseteq S$, let $S_{A,B}$ be the structure extending $S$ to this new language, where $P_1$ is interpreted as belonging to $A$ and $P_2$ is interpreted as belonging to $B$. To simplify notations, we will write $P_1$ as ``$\in A$'' and $P_2$ as ``$\in B$'' where possible.

\begin{lemma} \label{lem:definject}
    There exists a first--order sentence $\Psi$ in the extended language described above satisfying the following properties:
    \begin{itemize}
        \item if $A\subseteq B$, then $S_{A,B}\models\Psi$;
        \item if $A$ is a finite subset of $S$ and $B$ is any subset of $S$, then
        \[
            S_{A,B}\models \Psi \Leftrightarrow \left|A\right|\le\left|B\right|;
        \] 
        \item if $A$ is infinite and $B$ is finite, then $S_{A,B}\not\models \Psi$.
    \end{itemize} 
\end{lemma}

\begin{proof}
    We work in the structure $S_{A,B}$ in order to write the predicates as ``$\in A$'' and  ``$\in B$'' as described above. Consider the following auxiliary sentences:
    \[
        \operatorname{NZU}(x)\,:\,\forall\,y\,\left[xy\neq 1\,\wedge\,\left(x\neq 0\right)\right]  
    \]
    which holds for $x$ if and only if $x$ is neither zero nor a unit; and
    \[
        \operatorname{CM}(x,y)\,:\,\exists\,u,v\,\left(ux+vy=1\right)
    \]
    which holds for $x,y$ if and only if $(x)+(y)=S$.
    
    Let $\Psi_0$ be the sentence stating that there exist elements $g,m,c,d,u,g',m'$ such that all the following statements hold:
    \begin{itemize}
        \item $u\neq 0$
        \item $c\not\in A$
        \item $d\not\in B$
        \item $\forall\,x\in A\,\left(\operatorname{NZU}(1+g(x-c))\right)$
        \item $\forall\,x_1,x_2\in A\,\Big(x_1\neq x_2\,\rightarrow\,\operatorname{CM}(1+g(x_1-c), 1+g(x_2-c))\Big)$
        \item $\forall\,x_1,x_2\in A\,\Big(x_1\neq x_2\,\rightarrow\,(x_1-x_2)\mid u\Big)$
        \item $\forall\,x\in A\,\exists\,y\in B$ such that
        \begin{itemize}
            \item[$\triangleright$] $1+g(x-c)\mid m-y$; and
            \item[$\triangleright$] $\operatorname{NZU}(1+g'(y-d))$; and
            \item[$\triangleright$] $1+g'(y-d)\mid m'-x$; and
            \item[$\triangleright$] $1+g'(y-d)\nmid u$
        \end{itemize}
    \end{itemize}

    Suppose $\Psi_0$ holds. Let $f:A\to B$ be given by the final bullet point in that main list: for $x\in A$, $y=f(x)$ satisfies the statements of the sublist. We will show that $f$ is injective. Indeed, suppose $y=f(x_1)=f(x_2)$. Then $1+g'(y-d)$ divides both $m'-x_1$ and $m'-x_2$, hence it divides $x_1-x_2$. But, as $1+g'(y-d)$ does not divide $u$ and $u$ is divisible by every difference between elements of $A$, then $x_1=x_2$.

    Let
    \[
        \Psi : \Psi_0\,\vee\,\forall\,x\,\Big(x\in A\,\rightarrow\, x\in B\Big)\,\vee\,\forall\,x\Big(x\in B\Big).
    \]

    By the injectivity described above, this sentence $\Psi$ satisfies most of the statements of this Lemma. It only remains to show that if $A$ is finite and $|A|\le |B|$, with $B\neq S$, then $S_{A,B}\models\Psi$.

    Let $c,d$ be elements not in $A$, $B$ respectively. Let $u$ be a nonzero, nonunit divisible by all the differences between elements of $A$ ($u$ any nonzero, nonunit if $|A|\le 1$). Let $B_0$ be a subset of $B$ of cardinality $|A|$. Let $(a_i)$ be an enumeration of the elements of $A$; let $(b_i)$ be an enumeration of the elements of $B_0$.
    
    Let $g$ be such that the elements $1+g(a_i-c)$ are nonzero, nonunits, and pairwise comaximal; let $g'$ be such that the elements $u$ and $1+g'(b_i-d)$ are nonzero, nonunits and pairwise comaximal. By the Chinese Remainder Theorem, there exists $m$ such that $1+g(a_i-c)\mid m-b_i$; and there exists $m'$ such that $1+g'(b_i-d)\mid m'-a_i$. Hence $S_{A,B}\models \Psi$. 
\end{proof}

If $\mathcal{F}$, $\mathcal{G}$ are two uniformly parametrizable families of subsets of $S$, with parameter sets $C_1$ and $C_2$ respectively, then we can adapt the sentence $\Psi$ constructed in the proof to obtain a formula $\psi_{\mathcal{F}_1, \mathcal{F}_2}(\bar{x},\bar{y})$, in the original language and with arguments in $C_1\times C_2$, such that if $F_{\bar{x}}\in\mathcal{F}$ is the definable set corresponding to the parameters $\bar{x}$ and $G_{\bar{y}}\in\mathcal{G}$ is the definable set corresponding to the parameters $\bar{y}$, then
\begin{itemize}
    \item if $F_{\bar{x}}\subseteq G_{\bar{y}}$, then $S\models \psi_{\mathcal{F}_1, \mathcal{F}_2}(\bar{x},\bar{y})$;
    \item if $F_{\bar{x}}$ is a finite set, then $S\models\psi_{\mathcal{F}_1, \mathcal{F}_2}(\bar{x},\bar{y})$ if and only if $\left|F_{\bar{x}}\right|\le \left|G_{\bar{y}}\right|$; and
    \item if $F_{\bar{x}}$ is an infinite set and $G_{\bar{y}}$ a finite one, then $S\not\models \psi_{\mathcal{F}_1, \mathcal{F}_2}(\bar{x},\bar{y})$.
\end{itemize}

To make the notation clearer and less cumbersome we will write $\psi(F,G)$ for $F\in\mathcal{F}$ and $G\in\mathcal{G}$, and we will quantify over sets in a definable family as a shorthand for quantifying over parameters in their parameter set.

We will require this additional property later:

\subsubsection{$\Psi$-equivalence}

\begin{definition} \label{def:cardequiv}
    Let $\mathcal{F}$, $\mathcal{G}$ be two uniformly parametrizable families of subsets of $S$. We define the relation $\sim$ between these two families as
    \[
        F\sim G\,\Leftrightarrow\,\psi(F,G)\,\wedge\,\psi(G,F).
    \]

    Note that this relation is first--order definable (on the parameters used to define the families).
\end{definition}

Note that, despite the name, this relation is not necessarily an equivalence relation, as it may fail to be transitive.

\begin{definition} \label{def:Cequiv}
    The uniformly parametrizable family $\mathcal{C}_{\rm equiv}$ is defined from $\mathcal{C}_0$ via the following condition:
    \[
        \mathcal{C}_{\rm equiv} = \left\{F\in\mathcal{C}_0\,:\,\forall G,H\in\mathcal{C}_0\,\left(F\sim G\,\wedge\,F\sim H\,\rightarrow\,G\sim H\right)\right\}.
    \]    
    \end{definition}

Some remarks about this definition:
\begin{itemize}
    \item This is first--order definable: it will restrict the parameter set of $\mathcal{C}$ (definably) while keeping the same formula of definition.
    \item The relation $\sim$ is an equivalence relation on $\mathcal{C}_{\rm equiv}$.
    \item By the construction of $\Psi$, $\sim$ over the finite sets is just equipotence (equal cardinality). Thus $\mathcal{C}_{\rm equiv}$ contains the same finite sets as $\mathcal{C}$.
\end{itemize}

\subsubsection{The graph of addition}

\begin{definition}
    Let $\mathcal{F}$ be a parametrized family of definable subsets of $S$. We define the ternary relation $\Gamma(+)$ on $\mathcal{F}$ by having $\left(F_1,F_2,F_3\right)\in\Gamma(+)$ if and only if
    \[
        \exists x\,\left(F_1\cap (F_2+x) = \emptyset \,\wedge\, \left(F_1\cup \left(F_2+x\right)\sim F_3\right)\right).
    \]
\end{definition}

A few observations about this relation:
\begin{itemize}
    \item For any set $X$, $(X,\emptyset,X)\in\Gamma(+)$ and $(\emptyset, X, X)\in\Gamma(+)$.
    \item For a fixed family $\mathcal{F}$ this relation is first--order definable, as quantification over sets can be replaced by quantification over parameters.
    \item The role of the element $x$ is to shift the set $F_2$ so that the union of $F_1$ and $F_2+x$ is disjoint. In particular, if $F_1$ and $F_2$ are finite then $(F_1,F_2,F_3)\in\Gamma(+)$ if and only if $\left|F_1\right| + \left|F_2\right| = \left|F_3\right|$.
    \item By the previous bullet point, this relation respects $\sim$ if all the sets involved are finite. It might not do so if some of the sets are infinite.
\end{itemize}

\begin{definition} \label{def:equivplus}
    We define the binary relation $\equiv_+$ on $\mathcal{C}_{\rm equiv}$ as follows: $X\equiv_+ Y$ if and only if there exist $A,A',B,B'\in\mathcal{C}_{\rm equiv}$ such that
    \[
        A\sim A'\,\wedge\, B\sim B'\,\wedge\, (A,B,X)\in\Gamma(+)\,\wedge\,(A',B',Y)\in\Gamma(+).
    \]
\end{definition}

Clearly, this relation is first--order definable on the parameter set. It is coarser than $\sim$: if $X\sim Y$, then $(X,\emptyset, X)\in\Gamma(+)$ and $(Y,\emptyset, Y)\in\Gamma(+)$, thus $X\equiv_+ Y$. Moreover, for finite sets $\sim$ and $\equiv_+$ coincide.

\begin{definition} \label{def:Cplus}
    The uniformly parametrizable family $\mathcal{C}_{+}$ is defined from $\mathcal{C}_{\rm equiv}$ via the following condition:
    \[
        \mathcal{C}_{+} = \left\{F\in\mathcal{C}_{\rm equiv}\,:\,\forall G\in\mathcal{C}_{\rm equiv}\,\left(F\sim G\,\leftrightarrow\,F\equiv_+ G\right)\right\}.
    \]   
\end{definition}

\subsubsection{The graph of multiplication}

\begin{definition}
    Let $\mathcal{F}$ be a parametrized family of definable subsets of $S$. We define the ternary relation $\Gamma(\times)$ on $\mathcal{F}$ by having $\left(F_1,F_2,F_3\right)\in\Gamma(\times)$ if and only if there exist elements $\alpha,\beta$ satisfying
    \begin{itemize}
        \item $\operatorname{NZU}(\alpha)$; and
        \item $\operatorname{NZU}(\beta)$; and
        \item $\operatorname{CM}(\alpha,\beta)$; and
        \item $\forall\, x_1,x_2\in F_1\,\Big(x_1\neq x_2\,\rightarrow\,\operatorname{CM}(\alpha, x_1-x_2)\,\wedge\,\operatorname{CM}(\beta, x_1-x_2)\Big)$; and
        \item $\forall\, y_1,y_2\in F_2\,\Big(y_1\neq y_2\,\rightarrow\,\operatorname{CM}(\alpha, y_1-y_2)\,\wedge\,\operatorname{CM}(\beta, y_1-y_2)\Big)$; and
        \item $\left\{\beta y + \alpha x\,:\,x\in F_1\,\wedge\,y\in F_2\right\} \sim F_3$.
    \end{itemize}
\end{definition}

 Note that, as $S$ satisfies the conditions of Theorem \ref{thm:henson}, then such $\alpha$, $\beta$ can always be found when $F_1$, $F_2$ are finite. Moreover, similarly to the case of $\Gamma(+)$ and addition of cardinalities, this relation $\Gamma(\times)$ and multiplication of cardinalities coincide for finite sets:

\begin{proposition}
    Let $(F_1,F_2,F_3)\in\Gamma(\times)$ with $F_1$ and $F_2$ finite. Then $|F_1||F_2| = |F_3|$.
\end{proposition}

\begin{proof}
    By the nice properties of $\sim$ relative to finite sets, it suffices to show that the set $\left\{\beta y + \alpha x\,:\,x\in F_1\,\wedge\,y\in F_2\right\}$ constructed with $\alpha,\beta$ as in the definition of $\Gamma(\times)$ has cardinality $|F_1||F_2|$. In order to do so, we must show that if $(x,y),(x',y')\in F_1\times F_2$ are distinct pairs, then $\beta y + \alpha x \neq \beta y' + \alpha x'$.

    Suppose otherwise. As $S$ is a domain, this implies both $x-x'\neq 0$ and $y-y'\neq 0$. As $(\alpha,\beta)=(1)$ and $(\alpha, y-y')=(1)$, then $(\alpha,\beta(y-y'))=(1)$. But our supposition implies that $\beta(y-y')=-\alpha(x-x')$, hence $(\alpha)=(1)$. Contradiction.
\end{proof}

As was the case for $\Gamma(+)$, this relation is first--order definable on the parameter set of the parametrizable family.

\begin{definition} \label{def:equivtimes}
    We define the binary relation $\equiv_\times$ on $\mathcal{C}_{+}$ as follows: $X\equiv_\times Y$ if and only if there exist $A,A',B,B'\in\mathcal{C}_{+}$ such that
    \[
        A\sim A'\,\wedge\, B\sim B'\,\wedge\, (A,B,X)\in\Gamma(\times)\,\wedge\,(A',B',Y)\in\Gamma(\times).
    \]
\end{definition}

Note that this relation is first--order definable on the parameter set. For finite sets, $\sim$ and $\equiv_\times$ coincide.

\begin{definition} \label{def:Cplustimes}
    The uniformly parametrizable family $\mathcal{C}_{+,\times}$ is defined from $\mathcal{C}_{+}$ via the following condition:
    \[
        \mathcal{C}_{+,\times} = \left\{F\in\mathcal{C}_{+}\,:\,\forall G\in\mathcal{C}_{+}\,\left(F\sim G\,\leftrightarrow\,F\equiv_\times G\right)\right\}.
    \]
\end{definition}

Note the following:
\begin{itemize}
    \item This is first--order definable: it will restrict the parameter set of $\mathcal{C}_{+}$ (definably) while keeping the same formula of definition.
    \item $\mathcal{C}_{+,\times}$ contains the same finite sets as $\mathcal{C}_{+}$.
    \item $\sim$ is an equivalence relation on $\mathcal{C}_{+,\times}$.
    \item $\sim$ is compatible with $\Gamma(\times)$ in the following sense: if $A\sim A'$, $B\sim B'$ and $(A,B,X),(A',B',Y)\in\Gamma(\times)$, then $X\sim Y$.
    \item As $\mathcal{C}_{+,\times}$ is a subset of $\mathcal{C}_{+}$, $\sim$ is still compatible with $\Gamma(+)$ as mentioned above.
\end{itemize}

About this definition:
\begin{itemize}
    \item This is first--order definable: it will restrict the parameter set of $\mathcal{C}_{\rm equiv}$ (definably) while keeping the same formula of definition.
    \item $\mathcal{C}_{+}$ contains the same finite sets as $\mathcal{C}_{\rm equiv}$.
    \item $\sim$ is an equivalence relation on $\mathcal{C}_{+}$.
    \item $\sim$ is compatible with $\Gamma(\times)$ in the following sense: if $A\sim A'$, $B\sim B'$ and $(A,B,X),(A',B',Y)\in\Gamma(\times)$, then $X\sim Y$.
\end{itemize}

\subsubsection{The bound restriction}

\begin{definition}
    The uniformly parametrizable family $\mathcal{M}$ is defined from $\mathcal{C}_{+,\times}$ via the following condition:
    \begin{align}
        \mathcal{M} = \left\{F\in\mathcal{C}_{+,\times}\,: \,\forall G,H\in\mathcal{C}_{+,\times}\,\Big(\psi(G,F)\,\wedge\,\psi(H,F)\,\wedge\,\psi(G,H)\, \right. \nonumber \\
        \left. \rightarrow\,\exists K\in\mathcal{C}_{+,\times}\,\left(\psi(K,F)\,\wedge\,(K,G,H)\in\Gamma(+)\right)\Big)\right\}. \nonumber
    \end{align}

\end{definition}

This is first--order definable: it will restrict the parameter set of $\mathcal{C}_{+,\times}$ (definably) while keeping the same formula of definition.

As our original family $\mathcal{C}$ contains sets of every finite cardinality, for finite $F$ such sets $K$ always exist (they are defined by the condition $|K|=|H|-|G|$). Thus $\mathcal{M}$ contains the same finite sets as $\mathcal{C}_{+,\times}$.

\begin{definition} \label{def:Cbound}
    The uniformly parametrizable family $\mathcal{C}_{\rm bound}$ is defined as a subfamily of $\mathcal{C}_{+,\times}$ via the following condition:
    \[
        \mathcal{C}_{\rm bound} = \left\{F\in\mathcal{C}_{+,\times}\,:\,\exists M\in\mathcal{M}\,\Big(\psi(F,M)\Big)\right\}.
    \]

    This is first--order definable: it will restrict the parameter set of $\mathcal{C}_{+,\times}$ (definably) while keeping the same formula of definition. As such, the relation $\sim$ will be an equivalence relation on $\mathcal{C}_{\rm bound}$.
\end{definition}

As $\mathcal{C}_{+,\times}$ contains sets of every finite cardinality, $\psi$ behaves like the ``less or equal cardinality than'' relation for finite sets and $\Gamma(+)$ respects addition of cardinalities over finite sets, then $\mathcal{C}_{\rm bound}$ contains the same finite sets as $\mathcal{C}_{+,\times}$.

\begin{proposition} \label{prop:finite_substraction}
    Let $A$ be a finite set and $B$ be an infinite one, both of them in $\mathcal{C}_{\rm bound}$. There exists $K\in\mathcal{C}_{\rm bound}$ such that $(K,A,B)\in\Gamma(+)$.
\end{proposition}

\begin{proof}
    \begin{itemize}
        \item $\psi(A,B)$ (basic property of $\Psi$)
        \item there exists $M\in\mathcal{M}$ such that $\psi(B,M)$ (definition of $\mathcal{C}_{\rm bound}$)
        \item $M$ infinite (otherwise $\psi(B,M)$ would be false)
        \item $\psi(A,M)$ (basic property of $\Psi$)
        \item there exists $K\in\mathcal{C}_{+,\times}$ such that $\psi(K,M)$ and $(K,A,B)\in\Gamma(+)$ (definition of $\mathcal{M}$)
        \item $K\in\mathcal{C}_{\rm bound}$ (definition of $\mathcal{C}_{\rm bound}$). 
    \end{itemize}
\end{proof}

\subsubsection{Operations on the set of classes}

\begin{definition}
    The uniformly parametrizable family $\overline{\mathcal{C}}$ is defined as $\overline{\mathcal{C}} = \mathcal{C}_{\rm bound}\cup\{S\}$.
\end{definition}

As $S\not\in \mathcal{C}_{\rm bound}$, the union that defines the family $\overline{\mathcal{C}}$ is disjoint. Because of this, we can extend many of our definitions, in a first--order definable manner, from $\mathcal{C}_{\rm bound}$ to $\overline{\mathcal{C}}$:

\begin{definition} \label{def:Cbar}
    \begin{itemize}
        \item We set $S\sim S$, and for every $X\in \mathcal{C}_{\rm bound}$ we set $X\not\sim S$, $S\not\sim X$. Thus $\sim$ is an equivalence relation on $\overline{\mathcal{C}}$.
        \item For every $X\in\overline{\mathcal{C}}$, we set $(X,S,S)\in\Gamma(+)$ and $(S,X,S)\in\Gamma(+)$. The relation $\Gamma(+)$ is compatible with $\sim$ in the way stated before.
        \item We set $(\emptyset,S,\emptyset)\in\Gamma(\times)$ and $(S,\emptyset,\emptyset)\in\Gamma(\times)$. For $X\in\overline{\mathcal{C}}$, $C\neq\emptyset$, we set $(X,S,S)\in\Gamma(\times)$ and $(S,X,S)\in\Gamma(\times)$. The relation $\Gamma(\times)$ is compatible with $\sim$ in the way stated before.
    \end{itemize}
\end{definition}

\begin{definition}
    Let $\mathcal{N}$ be the quotient set of $\overline{\mathcal{C}}$ modulo $\sim$. We define two binary operations, $+$ and $\times$, on $\mathcal{N}$:
    \begin{itemize}
        \item If there exist sets $A_0\in [A]$, $B_0\in [B]$ and $C_0\in\overline{\mathcal{C}}$ such that $(A_0,B_0,C_0)$ is in $\Gamma(+)$, then we set $[A]+[B] = [C_0]$. By the compatibility of $\Gamma(+)$ with $\sim$, this is well-defined.
        \item Otherwise, define $[A]+[B]=[S]$.
        \item Similarly for $\times$: $[A]\times [B]=[C_0]$ if $A_0\in [A]$, $B_0\in [B]$ and $(A_0,B_0,C_0)\in\Gamma(\times)$, this is well-defined due to compatibility; otherwise $[A]\times [B] = [S]$.
    \end{itemize}
\end{definition}

\begin{definition}
    Let $[1]$ be the class of the singleton sets. The successor function in $\mathcal{N}$ is defined as $S([A])=[A]+[1]$.
\end{definition}

\begin{proposition}
    $\left(\mathcal{N};[\emptyset],S,+,\times\right)$ is a model of the theory $R$ of Tarski, Mostowski, and Robinson.
\end{proposition}

\subsection{Applying the criterion to the present work}

Let $p$ be an odd prime and let $q$ be a power of $p$. Let $R=\mathbb{F}_{q}[t], K=\mathbb{F}_{q}(t)$, and let $\|x\|_\infty:=q^{\deg(b)-\deg(a)}$  where $x=\frac{a}{b}$ with $a\in R$ and $b\in R\setminus\{0\}$. Denote by $K^{\rm sep}$ a fixed separable closure of $K$ and fix a norm on $K^{\rm sep}$ extending the one in $K$, which we will also  denote it by  $\| \ \|_\infty$ by abuse of notation. Let $\mathcal{O}_{K^{\rm sep}}$ be the integral closure of $R$ in $K^{\rm sep}$. We will show that many subrings $S$ of  $\mathcal{O}_{K^{\rm sep}}$ satisfy the hypotheses of Theorem~\ref{thm:henson}.

\begin{lemma} \label{divisibilidades}
    Let $S$ be any subring of $\mathcal{O}_{K^{\rm sep}}$ containing $R$. Then the following holds:
     \begin{enumerate}[(a)]
            \item \label{item:chinese}For any finite set $A\subseteq S$ there exists an element $x$, neither zero nor a unit, such that $(x)+(a_i)=S$ for every $a_i\in A\setminus\left\{0\right\}$.
            \item \label{item:comax}For any finite set $A\subseteq S$ that does not contain zero and any $a\in S$ that is neither a zero nor a unit there exists an element $g\in S$ such that:
            \begin{itemize}
                \item for any $a_i\in A$, $1+a_ig$ is neither zero nor a unit;
                \item for any $a_i\in A$, $(a)+(1+a_ig) = S$; and
                \item for distinct $a_i,a_j\in A$, $(1+a_ig) + (1+a_jg) = S$.
            \end{itemize}
        \end{enumerate}
\end{lemma}

\begin{proof}
Let $s\in S\setminus\{0\}$ be given. Since $s$ is integral over $R$, it follows that the minimal polynomial of $s$, $p_s(x):=x^n+c_{n-1}x^{n-1}+\dots+c_0$, belongs to $R[x]$ and the constant coefficient $c_0$ is different from $0$. It follows that $-c_0=s(s^{n-1}+c_{n-1}s^{n-1}+\dots+c_1).$ Thus, $sS\cap R$ contains a non-zero element. We now proceed to prove clause (\ref{item:chinese}): let $a_1,\dots, a_n$ be an enumeration of the non-zero elements of $A$. For each $i$, pick a non-zero multiple $b_i$ of $a_i$ which belongs to $R$. If $x$ is any irreducible polynomial in $R$ coprime with the $b_i$'s, then $x$ is as required.

Let us now prove clause (\ref{item:comax}). Let $L=K(a_1,\dots,a_n)$ and let $\sigma_1,\dots,\sigma_k$ be the embeddings of $L$ into $K^{\rm sep}$ fixing $K$. Let $c$ be a nonzero multiple of $\prod\limits_{i\ne j}(a_i-a_j)^2$ which belongs to $R$ (if $n=1$, then set $\prod\limits_{i\ne j}(a_i-a_j)^2=1$). Pick $M\in R,$ of large enough degree, so that $\| \sigma_j(1+a_iaMc)\|_\infty>1$ for any $j=1,\dots,k$ and any $i=1,\dots,n.$ We claim that $g=aMc$ is as required. First of all, note that $1+a_ig$ are neither zero nor units as $\|\sigma_1(1+a_ig)\|_\infty\cdots \|\sigma_1(1+a_ig)\|_\infty= \| N_{L/K}(1+a_ig)\|_\infty>1$ for any $i=1,\dots,n.$ Notice that since $a | g$, then clearly $I+I_i=S$. We are left to show the last bullet point. Let $i\ne j$ be given. Observe that
\[
    (a_i-a_j)g=(1+a_ig)-(1+a_jg)\in I_i+I_j
\]
By our choice of $g$ there is an $s\in S$ such that $s(a_i-a_j)g=g^2$. Thus, $g^2\in I_i+I_j$. It follows that 
\[
    1+2a_ig=(1+a_ig)^2-a_i^2g^2
\]
also belongs to $I_i+I_j$. Finally, we have that
\[
    1=2(1+a_ig)-(1+2a_ig)\in I_i+I_j
\]
as required.
\end{proof}	
		
\begin{corollary} \label{coro:hensonapplied}
    Let $S$ be any subring of $\mathcal{O}_{K^{\rm sep}}$ containing $R$. Let $\mathcal{F}$ be a uniformly parametrizable family of subsets of $R$ which contains finite sets of arbitrary large size, then there is an interpretation of the Theory $R$ in the first-order structure of the ring $S$. In particular, the theory of $S$ is undecidable.
\end{corollary}

\section{Definability of the ring of integers} \label{sec:sasha}

\subsection{Some results from algebraic number theory} \label{ant}

In this section we review some basic results from algebraic number theory and also some more specific results from the appendix in \cite{sasha}.

Let $K$ be a global function field. By a prime of $K$ we mean an equivalence class of nontrivial absolute values of $K$. In a global function field, all such absolute values are non-Archimedean and we can represent a prime as a pair $(\mathfrak{p}, \mathcal{O}_\mathfrak{p})$ where $\mathcal{O}_\mathfrak{p}$ is a valuation ring of $K$ and $\mathfrak{p}\subseteq \mathcal{O}_\mathfrak{p}$ is its unique maximal ideal. We will often refer only to the ideal $\mathfrak{p}$ as a prime of $K$. We denote by $v_\mathfrak{p}$ the associated normalized discrete valuation on $K$. 

Define $K_\mathfrak{p}$ to be the completion of $K$ at $\mathfrak{p}$, $R_\mathfrak{p}$ the ring of integers in $K_\mathfrak{p}$, the maximal ideal of $R_\mathfrak{p}$ by $\hat{\mathfrak{p}}$, and by $\mathbb{F}_\mathfrak{p}$ the residue field of $\hat{\mathfrak{p}}$, and let $\operatorname{red}_\mathfrak{p}: R_\mathfrak{p}\to \mathbb{F}_\mathfrak{p}.$ Let $\mathcal{O}_\mathfrak{p}=R_\mathfrak{p}\cap K$.

If $L$ is a finite separable extension of $K$ and $\mathfrak{q}$ is a prime in $L$ above $\mathfrak{p}$, then the ramification degree $e(\hat{\mathfrak{q}},\hat{\mathfrak{p}})$ is the exponent in the factorization $\hat{\mathfrak{p}}R_\mathfrak{q}=\hat{\mathfrak{q}}^e$ and the extension degree $f(\hat{\mathfrak{q}},\hat{\mathfrak{p}})$ is $[\mathbb{F}_\mathfrak{q} : \mathbb{F}_\mathfrak{p}]$. It is well known that $R_\mathfrak{p}/\hat{\mathfrak{p}}\cong \mathcal{O}_\mathfrak{p}/\mathfrak{p}$ and that the residue field degrees and the ramification indices coincide, i.e., $f(\mathfrak{q},\mathfrak{p})=f(\hat{\mathfrak{q}},\hat{\mathfrak{p}})$ and $e(\mathfrak{q},\mathfrak{p})=e(\hat{\mathfrak{q}},\hat{\mathfrak{p}})$.

Let $\ell$ be a rational prime coprime with the characteristic of $K$ and let $\zeta_\ell$ be a primitive $\ell$-th root of unity.
   
\begin{lemma}\label{lemma:rosen}
    Let $\zeta_\ell, b\in K$ with $b$ not an $\ell$-th power, and let $\mathfrak{p}$ a prime of $K$. The following hold:
    \begin{enumerate}[{\rm (i)}]
        \item \label{item:rosen_1}The prime $\mathfrak{p}$ is ramified in $K(\sqrt[\ell]{b})$ if and only if $\ncongruente{v_{\mathfrak{p}}(b)}{0}{\ell}.$
        \item If $v_{\mathfrak{p}}(b)=0$, then $\mathfrak{p}$ splits in $F(\sqrt[\ell]{b })$ if and only if $b\in  \mathbb{F}_\mathfrak{p}^{\times\ell} $.
    \end{enumerate}
\end{lemma}

\begin{proof}
    See Propositions 10.3 and 10.5 in \cite{rosen}.
\end{proof}
 	        
\begin{lemma}\label{lemma:rumely}
    If $\zeta_\ell, b\in K$ with $b$ not an $\ell$-th power, $\mathfrak{p}$ a prime of $K$, and $\mathfrak{q}$ a prime of $L:=K(\sqrt[\ell]{b})$ above $\mathfrak{p}$, then the following hold:
    \begin{enumerate}[{\rm (i)}]
        \item If $v_{\mathfrak{p}}(b)\not\equiv\ 0 \ {\rm mod}\ \ell$, then $L_{\mathfrak{q}}/K_{\mathfrak{p}}$ is totally ramified  and $\norm{L_\mathfrak{q}}{K_\mathfrak{p}}{L_\mathfrak{q}^\times}$ is generated by $b$ and $(K_{\mathfrak{p}}^\times)^\ell.$
        \item \label{item:rumely_2}If $v_{\mathfrak{p}}(b)\equiv\ 0 \ {\rm mod}\ \ell$ but $b\notin \mathbb{F}^{\times\ell}$, then $L_\mathfrak{q}/ K_\mathfrak{p}$ is unramified and of degree $\ell$ and $\norm{L_\mathfrak{q}}{K_\mathfrak{p}}{L_\mathfrak{q}^\times }=\{x\in K_\mathfrak{p}^\times : v_\mathfrak{p}(x)\equiv\ 0 \ {\rm mod}\ \ell \}$. 
        \item \label{item:rumely_3}If $\congruente{v_\mathfrak{p}(b)}{0}{\ell}$ and $b\in K_\mathfrak{p}{^\times\ell}$, then $L_\mathfrak{q} / K_\mathfrak{p}$ is trivial  and $\norm{L_\mathfrak{q}}{K_\mathfrak{p}}{L_\mathfrak{q}^\times }=K_\mathfrak{p}^\times$.
    \end{enumerate}
\end{lemma}

\begin{proof}
        See Lemma 1 in \cite{rumely}.
\end{proof}

Let $x, b, d \in K^*$ be such that $bx^\ell+b^\ell\ne 0$ and $d+d^{-1}\ne 0.$ Define 
\begin{itemize}
    \item  $L_1:=K\left (\sqrt[\ell]{1+x^{-1}} \right ).$
    \item  $L_2:=K\left (\sqrt[\ell]{1+x^{-1}},\sqrt[\ell]{1+(bx^\ell+b^\ell)^{-1}} \right ).$
    \item  $L:=K\left (\sqrt[\ell]{1+x^{-1}},\sqrt[\ell]{1+(bx^\ell+b^\ell)^{-1}},\sqrt[\ell]{1+(d+d^{-1})x^{-1}} \right ).$
\end{itemize}

\begin{lemma}\label{lemma:ncong1}   
    If $\zeta_\ell\in K$, and $\mathfrak{p}$ is a prime of $K$ which satisfies the following conditions: 
    \begin{enumerate}[{\rm (i)}]
        \item the valuation $v_{\mathfrak{p}}(c)=0$ and $\operatorname{red}_{\mathfrak{p}}(c)\notin \mathbb{F}_\mathfrak{p}^{\times\ell}$;
        \item $v_{\mathfrak{p}}(x)<0;$
        \item $\ncongruente{v_{\mathfrak{p}}(b)}{0}{\ell}$;
        \item $v_{\mathfrak{p}}(bx^\ell)< v_{\mathfrak{p}}(b^\ell)$ and
        \item \label{item:negative_v}$v_{\mathfrak{p}}(b)<0$,
    \end{enumerate}
    then for any prime factor $\mathfrak{q}$ of $\mathfrak{p}$ in $L$ we have that $\operatorname{red}_{\mathfrak{q}}(c)\notin \mathbb{F}_\mathfrak{q}^{\times\ell}$ and $\ncongruente{v_{\mathfrak{q}}(bx^\ell+b^\ell)}{0}{\ell}$. 
\end{lemma}

\begin{proof}
    See Proposition 7.9, \cite{sasha}; the arguments there essentially work in our setting. Just notice that there are no primes above $\ell,$ and clause (\ref{item:negative_v}) is necessary to ensure that $v_{\mathfrak{p}}(bx^\ell+b^\ell)<0.$ 
\end{proof}

\begin{lemma}\label{lemma:congruentes1}
    If $\zeta_\ell\in K$, then for any prime $\mathfrak{q}$ of $L$ which is not a pole of $x$ the following hold:
    \begin{enumerate}[{\rm (i)}]
        \item $\congruente{v_\mathfrak{q}(x)}{0}{\ell};$
        \item $\congruente{v_\mathfrak{q}( bx^\ell+b^\ell)}{0}{\ell};$
        \item $\congruente{v_\mathfrak{q}(c)}{0}{\ell}.$
    \end{enumerate}
\end{lemma}

\begin{proof}
    See Proposition 7.10 in \cite{sasha}. 
\end{proof}

Let $x, d, a \in K^*$ be such that $dx^\ell+d^\ell\ne 0$ and $a+a^{-1}\ne 0$. Define 
\begin{itemize}
    \item  $H_1:=K\left (\sqrt[\ell]{1+d^{-1}} \right ).$
    \item  $H_2:=K\left (\sqrt[\ell]{1+d^{-1}},\sqrt[\ell]{1+(dx^\ell+d^\ell)^{-1}} \right ).$
    \item  $H:=K\left (\sqrt[\ell]{1+d^{-1}},\sqrt[\ell]{1+(dx^\ell+d^\ell)^{-1}},\sqrt[\ell]{1+(a+a^{-1})d^{-1}} \right ).$
\end{itemize}

The following Lemma plays a role similar to Proposition 7.11 of \cite{sasha}.

\begin{lemma}\label{lemma:ncong2}
    If $\zeta_\ell\in K$, and $\mathfrak{p}$ is a prime of $K$ which satisfies the following conditions: 
    \begin{enumerate}[{\rm (i)}]
        \item \label{item:ncong2_1}$v_{\mathfrak{p}}(d)<0;$
        \item \label{item:ncong2_2}$v_{\mathfrak{p}}(dx^\ell)< v_{\mathfrak{p}}(d^\ell)$;
        \item \label{item:ncong2_3}The valuation $v_{\mathfrak{p}}(a)=0$ and $\operatorname{red}_{\mathfrak{p}}(a)\notin \mathbb{F}_\mathfrak{p}^{\times\ell} $ and
        \item \label{item:ncong2_4}$\ncongruente{v_{\mathfrak{p}}(d)}{0}{\ell}.$
    \end{enumerate}

    Then for any prime factor $\mathfrak{q}$ of $\mathfrak{p}$ in $H$ we have that $\ncongruente{v_{\mathfrak{q}}(dx^\ell+d^\ell)}{0}{\ell}$ and $\operatorname{red}_{\mathfrak{q}}(c)\notin \mathbb{F}_\mathfrak{q}^{\times\ell}$.
\end{lemma}

\begin{proof}
    By assumption (\ref{item:ncong2_1}) we have that $v_\mathfrak{p}(1+d^{-1})=0$ and $1+d^{-1}\in \mathbb{F}_\mathfrak{p}^{\times\ell}.$ Therefore, by Lemma \ref{lemma:rosen}, $\mathfrak{p}$ splits completely in $H_1$. Let $\mathfrak{q}_1:=\mathfrak{q}\cap H_1$. By assumptions (\ref{item:ncong2_1}) and (\ref{item:ncong2_2}) we have that $v_{\mathfrak{q}_1}\left( 1+(dx^\ell+d^\ell)^{-1}\right )=0$ and $1+(dx^\ell+d^\ell)^{-1}\in \mathbb{F}_{\mathfrak{q}_1}^{\times\ell}.$ Therefore, by using Lemma \ref{lemma:rosen} once more, we obtain that $\mathfrak{q}_1$ splits in $H_2$. Let $\mathfrak{q}_2:=\mathfrak{q}\cap H_2$. By clauses (\ref{item:ncong2_1}) and (\ref{item:ncong2_3}) we have that $v_{\mathfrak{q}_2}\left (1+(a+a^{-1})d^{-1}\right )=0$ and $1+(a+a^{-1})d^{-1}\in \mathbb{F}_{\mathfrak{q}_2}^{\times\ell}.$ Applying Lemma \ref{lemma:rosen} again, we get that $\mathfrak{q}_2$ splits in $H$. Using the tower law we get that $e(\mathfrak{q},\mathfrak{p})=1$. Therefore, $v_{\mathfrak{q}}(dx^\ell+d^\ell)=v_{\mathfrak{p}}(dx^\ell+d^\ell)=v_\mathfrak{p}(dx^\ell),$ since the last value is not divisible by $\ell$ by clause (\ref{item:ncong2_4}), we obtain that $\ncongruente{v_{\mathfrak{q}}(dx^\ell+d^\ell)}{0}{\ell}$ as required. Finally, by the tower law we have that $f(\mathfrak{q},\mathfrak{p})=1$. It follows that $\mathbb{F}_\mathfrak{q}=\mathbb{F}_\mathfrak{p}$ and therefore $\operatorname{red}_{\mathfrak{q}}(c)\notin \mathbb{F}_\mathfrak{q}^{\times\ell}$. This finishes the proof of the Lemma.
\end{proof}

The following Lemma is similar to Proposition 7.12 of \cite{sasha}. However, since our assumptions are slightly more general, we provide a full proof.

\begin{lemma}\label{lemma:congruentes2}
If $\zeta_\ell\in K$, then for any prime $\mathfrak{q}$ of $H$ which is not a pole of $d$, the following hold:
    \begin{enumerate}[{\rm (i)}]
        \item $\congruente{v_\mathfrak{q}(d)}{0}{\ell};$
        \item \label{item:cong2_2}$\congruente{v_\mathfrak{q}( dx^\ell+d^\ell)}{0}{\ell}$ and
        \item $\congruente{v_\mathfrak{q}(a)}{0}{\ell}$.
    \end{enumerate}
\end{lemma}

\begin{proof}
    Let $\mathfrak{p}:=\mathfrak{q}\cap K.$ If $\congruente{v_\mathfrak{p}(d)}{0}{\ell}$, then also $\congruente{v_\mathfrak{q}(d)}{0}{\ell}.$ Thus, we may assume that $\ncongruente{v_\mathfrak{p}(d)}{0}{\ell}$. It follows from our assumptions that $v_\mathfrak{p}(1+d^{-1})=-v_\mathfrak{p}(d)<0,$ and this last expression is not divisible by $\ell$. It follows from clause (\ref{item:rosen_1}) of Lemma~\ref{lemma:rosen} that $\mathfrak{p}$ ramifies in $H_1$, and hence $\congruente{v_\mathfrak{q}(d)}{0}{\ell}$. 
    
    Proceeding as before let $\mathfrak{q}_1:= \mathfrak{q}\cap H_1$. We may assume, without loss of generality, that $\ncongruente{v_{\mathfrak{q}_1}( dx^\ell+d^\ell)}{0}{\ell}.$  By the previous argument we have that both $v_{\mathfrak{q}_1}( dx^\ell)$ and $v_{\mathfrak{q}_1}(d^\ell)$ are divisible by $\ell$, and it follows from the ultrametric inequality  that $v_{\mathfrak{q}_1}(dx^\ell)=v_{\mathfrak{q}_1}( d^\ell)$. Hence $v_{\mathfrak{q}_1}(dx^\ell+d^\ell)>\min\{v_{\mathfrak{q}_1}(dx^\ell),v_{\mathfrak{q}_1}(d^\ell)\}\geq 0$. Applying Lemma~\ref{lemma:rosen} again, we obtain that $\mathfrak{q}_1$ ramifies in $H_2$ and therefore clause (\ref{item:cong2_2}) holds. 
    
    Finally, let $\mathfrak{q}_2:=\mathfrak{q}\cap H_2$. We may assume, without loss of generality, that $\ncongruente{v_{\mathfrak{q}_2}(a)}{0}{\ell}$. Then
    \[
        v_{\mathfrak{q}_2}((a+a^{-1})d^{-1})<0 \  {\rm and} \ \congruente{v_{\mathfrak{q}_2}((a+a^{-1})d^{-1})}{v_{\mathfrak{q}_2}(a)}{\ell},
    \]
    where we use the fact that $\congruente{v_{\mathfrak{q}_1}(d)}{0}{\ell}$. Using Lemma~\ref{lemma:rosen} once more we obtain the desired result.
\end{proof}

\subsection{Defining the ring of $S$-integers using norm forms. }\label{norms}

Let $K$ be a global function field and let $K_{\rm inf}$ be an infinite separable algebraic extension of $K$. For each global field $N$ such that $K\subseteq N\subseteq K_{\rm inf}$ we denote by $I_N$ the set of all global field extensions of $N$ contained in $K_{\rm inf}$. We recall the following key definitions from \cite{sasha}. 

\begin{definition}
    Let $K$ be a global function field and let $K_{\rm inf}$ be an infinite separable algebraic extension of $K$. 
    \begin{itemize}
        \item We say that a prime $\mathfrak{p}_K$ of $K$ is $\ell$-{\bf bounded}, if there exist a finite extension $M$ of $K$ and a prime $\mathfrak{p}_M$ above it, so that for any $N\in I_M$ there is a prime $\mathfrak{p}_N$ above $\mathfrak{p}_M$ such that $\ncongruente{d(\mathfrak{p}_N,\mathfrak{p}_M)}{0}{\ell}$.
        \item We say that a prime $\mathfrak{p}_K$ of $K$ is {\bf hereditarily} $\ell$-{\bf bounded}, if for every $M\in I_K$, every prime $\mathfrak{p}_M$ in $M$ above $\mathfrak{p}_K$ is $\ell$-bounded.
         \item We say that a prime $\mathfrak{p}_K$ of $K$ is {\bf completely} $\ell$-{\bf bounded}, if there is a finite extension $M$ of $K$ and a prime $\mathfrak{p}_M$ above it, so that for every $N\in I_M$ and every prime $\mathfrak{p}_N$ above $\mathfrak{p}_M$ we have that $\ncongruente{d(\mathfrak{p}_N,\mathfrak{p}_M)}{0}{\ell}$.
    \end{itemize}

    In these definitions, $d(\mathfrak{p}_N,\mathfrak{p}_M)=e(\mathfrak{p}_N,\mathfrak{p}_M)f(\mathfrak{p}_N,\mathfrak{p}_M)$ denotes the local extension degree.
\end{definition}

Before proceeding any further we need to introduce some notation. Let $K$ be a global function field with $\zeta_\ell\in K$ and $\mathscr{S}_K$ be a non-empty finite set of primes and let $K_{\rm inf}$ be an infinite algebraic extension of $K$. For any $M\in I_K$ and any prime $\mathfrak{p}_K$ of $K$, let $\mathcal{C}_M(\mathfrak{p}_K)$ be the set of all primes in $M$ above $\mathfrak{p}_K$. We also denote by $\mathscr{S}_M:=\bigcup\limits_{\mathfrak{p}\in \mathscr{S}_K}\mathcal{C}_M(\mathfrak{p})$ and $\mathscr{S}_{{\rm inf}}:=\bigcup\limits_{M\in I_K}\mathscr{S}_M$. For each $M\in I_K$, let
\[
    \Theta(M,\mathscr{S}_M):=\{c\in M : \forall \mathfrak{p}\in \mathscr{S}_M [  v_\mathfrak{p}(c-1)>0]\}
\]
and let
\[
    \Theta(K_{\rm inf},\mathscr{S}_{{\rm inf}}):=\bigcup\limits_{M\in I_K} \Theta(M,\mathscr{S}_M).
\]

In the following Lemma we will use hereditarily $\ell$-boundedness to define being integral at all primes not in $\mathscr{S}_{{\rm inf}}$, using norm forms and the predicate $\Theta(K_{\rm inf},\mathscr{S}_{{\rm inf}})$.

\begin{lemma}\label{mainlema:sasha}
    Let $K$ be a global function field with $\zeta_\ell\in K$, $\mathscr{S}_K$ a non-empty finite set of primes and $K_{\rm inf}$ an infinite algebraic extension of $K$. Suppose that all primes of $K$ not in $\mathscr{S}_K$ are hereditarily $\ell$-bounded in $K_{\rm inf}$. Then
    \begin{align}
        \mathcal{O}_{K_{\rm inf},\mathscr{S}_{{\rm inf}}}=\{0\}\cup\{&x\in K_{\rm inf}^*\,:\,\forall c\in\Theta(K_{\rm inf},\mathscr{S}_{K_{\rm inf}})\,  \forall b\in K_{\rm inf} \nonumber\\
        &(bx^\ell +b^\ell=0 \vee \exists y \in L_{\rm inf}(\sqrt[\ell]{c})  [\norm{L_{\rm inf}(\sqrt[\ell]{c})}{L_{\rm inf}}{y}=bx^\ell+b^\ell ] ) \} \nonumber
    \end{align}
 \end{lemma}
 
 \begin{proof}
    Fix $x\in \mathcal{O}^*_{K_{\rm inf},\mathscr{S}_{{\rm inf}}}, c \in\Theta(K_{\rm inf},\mathscr{S}_{K_{\rm inf}}) $ and $ b\in K_{\rm inf}$ such that $bx^\ell+b^\ell\ne0$. We need to find $y\in L_{\rm inf}(\sqrt[\ell]{c})$ such that $\norm{L_{\rm inf}(\sqrt[\ell]{c})}{L_{\rm inf}}{y}=bx^\ell+b^\ell$.
    
    Fix $M\in I_K$ so that $x,b,c\in M$ and $[L_{\rm inf}(\sqrt[\ell]{c}): L_{\rm inf}]=[L_{M}(\sqrt[\ell]{c}): L_{M}]$ where
    \[
        L_M=M\left (\sqrt[\ell]{1+x^{-1}}, \sqrt[\ell]{1+(bx^\ell+b^\ell)^{-1}},  \sqrt[\ell]{1+(c+c^{-1})x^{-1}}  \right ).
    \]
    
    It follows that $\norm{L_{\rm inf}(\sqrt[\ell]{c})}{L_{\rm inf}}{y} =\norm{L_M(\sqrt[\ell]{c})}{L_M}{y}$ for any $y\in L_M(\sqrt[\ell]{c})$. Thus, it suffices to find $y\in L_M(\sqrt[\ell]{c})$ such that $\norm{L_M(\sqrt[\ell]{c})}{L_M}{y}=bx^\ell+b^\ell$. By Hasse's norm Theorem, it suffices to show that $bx^\ell+b^\ell$ is a norm locally for any prime of $L_M$. Let $\mathfrak{p}_{L_M}$ be a prime of $L_M$. First consider the case $\mathfrak{p}_{L_M}\notin  \mathscr{S}_{L_M}$. It follows from Lemma~\ref{lemma:congruentes1} that $\congruente{v_{\mathfrak{p}_{L_M}}(bx^\ell+b^\ell)}{0}{\ell}$ and $\congruente{v_{\mathfrak{p}_{L_M}}(c)}{0}{\ell}.$ Applying Lemma~\ref{lemma:rumely} we have that $bx^\ell+b^\ell$ is a norm in $L_{M,\mathfrak{p}_{L_M}}$ as required. Now, consider the case $\mathfrak{p}_{L_M}\in \mathscr{S}_{L_M}$. It follows from the choice of $c$ that $v_{\mathfrak{p}_{L_M}} (c-1)=0.$ Hence, we are in clause (\ref{item:rumely_3}) of Lemma~\ref{lemma:rumely} and the result also follows in this case.    
    
    Suppose now that $x\notin \mathcal{O}_{K_{\rm inf},\mathscr{S}_{{\rm inf}}}$. We need to find $c$ and $b$ such that the norm equation does not have a solution. Fix a prime $\mathfrak{p}_{K(x)}$ of $K(x)$ which does not belong to $\mathscr{S}_{K(x)}$ and such that $v_{\mathfrak{p}_{K(x)}} (x)<0.$ Since all the primes in $K$ not in $\mathscr{S}_K$ are hereditarily $\ell$-bounded, it follows that $\mathfrak{p}_{K(x)}$ is an $\ell$-bounded prime. So there is an $\ell$-bounding field $M\in I_{K(x)}$ and an $\ell$-bounding prime $\mathfrak{p}_M \in M$ above $\mathfrak{p}_{K(x)}$. By using the strong approximation theorem, we can choose $c\in M$ such that $v_{\mathfrak{p}} (c-1)>0$ for any $\mathfrak{p} \in \mathscr{S}_M$, $v_{\mathfrak{p}_M}(c)=0$ and $c$ is not an $\ell$-th power $\mathbb{F}_{\mathfrak{p}_M}$; and we can also choose $b\in M$ such that $v_{\mathfrak{p}_M}(b)=-1$.

    \textbf{Claim:} There is no $y\in L_{\rm inf}(\sqrt[\ell]{c})$ such that \begin{equation} \label{eqn:mainlemma:sasha}
        \norm{L_{\rm inf}(\sqrt[\ell]{c})}{L_{\rm inf}}{y}=bx^\ell+b^\ell.
    \end{equation} 

    Aiming towards a contradiction. Suppose that there is a solution $y$ to equation \ref{eqn:mainlemma:sasha}. Fix $N\in I_M$ such that $y\in L_N(\sqrt[\ell]{c})$ and $[L_N(\sqrt[\ell]{c}):L_N]=[L_{K_{\rm inf}}(\sqrt[\ell]{c}):L_{K_{\rm inf}}]$. Hence, $\norm{L_N(\sqrt[\ell]{c})}{L_N}{y}=bx^\ell+b^\ell$. Fix a prime $\mathfrak{p}_N\in \mathcal{C}_N(\mathfrak{p}_M)$ such that $\ncongruente{d(\mathfrak{p}_N,\mathfrak{p}_M)}{0}{\ell}$.

    Thus, we are in the following situation:
    \begin{enumerate}[{\rm (i)}]
        \item The valuation $v_{\mathfrak{p}_N}(c)=0$ and $\operatorname{red}_{\mathfrak{p}_N}(c)\notin \mathbb{F}_{\mathfrak{p}_N}^{\times\ell}$;
        \item $v_{\mathfrak{p}_N}(x)<0;$
        \item $\ncongruente{v_{\mathfrak{p}_N}(b)}{0}{\ell}$;
        \item $v_{\mathfrak{p}_N}(bx^\ell)< v_{\mathfrak{p}}(b^\ell)$ and
        \item $v_{\mathfrak{p}_N}(b)<0$.
    \end{enumerate}

    These are the assumptions of Lemma \ref{lemma:ncong1}, therefore $v_{\mathfrak{p}_{L_N}}(c)=0$, $c$ is not an $\ell$-th power in $L_N$, and also $ \congruente{v_{\mathfrak{p}_{L_N}}(bx^\ell+b^\ell)}{-1}{\ell}$. This contradicts clause (\ref{item:rumely_2}) of Lemma \ref{lemma:rumely}. 
\end{proof}

We shall find a definition of $\Theta(K_{\rm inf},\mathscr{S}_{{\rm inf}})$ using norm forms. For this we will use completely $\ell$-boundedness to define the rings $\mathcal{O}_\mathfrak{p}$ for $\mathfrak{p}\in \mathscr{S}_K$.  

\begin{lemma}\label{mainlema2}
    Let $K$ be a global function field with $\zeta_\ell\in K$, $\mathscr{S}_K$ a finite non-empty set of primes and $K_{\rm inf}$ an infinite separable algebraic extension of $K$. Suppose that all primes in $\mathscr{S}_K$ are completely $\ell$-bounded in $K_{\rm inf}$. Let $M_\ell$ be a completely $\ell$-bounded field for all primes in $\mathscr{S}_K$. Fix $d\in M_\ell$ such that $v_\mathfrak{q}(d)<0,$ and $d$ does not have any other poles, and $\ncongruente{v_\mathfrak{q}(d)}{0}{\ell}$ for all $\mathfrak{q}\in \mathscr{S}_{M_\ell}$. Fix $a\in M_\ell$ such that $v_\mathfrak{q}(a)=0$ and $a$ is not an $\ell$-th power in $\mathbb{F}_\mathfrak{q}$ for all $\mathfrak{q}\in \mathscr{S}_{M_\ell}$. Let
    \[
        H_{\rm inf}:=K\left (\sqrt[\ell]{1+d^{-1}},\sqrt[\ell]{1+(dx^\ell+d^\ell)^{-1}},\sqrt[\ell]{1+(a+a^{-1})d^{-1}}\right)
    \]
    and let 
    \[
        B(K_{\rm inf},\ell,a,d):=\{ x\in K_{\rm inf} : \exists  y\in H_{\rm inf}(\sqrt[\ell]{a}) [ \norm{H_{\rm inf}(\sqrt[\ell]{a})}{H_{\rm inf}}{y}=dx^\ell+d^\ell]\}.
    \]
    Then
    \[
        B(K_{\rm inf},\ell,a,d)=\{x\in K_{\rm inf} : \forall K\in I_{M_\ell(x)}\, \forall \mathfrak{p}_K \in \mathscr{S}_K \, [ v_{\mathfrak{p}_K}(dx^\ell)>v_{\mathfrak{p}_K}(d^\ell)]  \}.
    \]
\end{lemma}

\begin{proof}
    Let $x\in K_{\rm inf}$ be given such that
    \[
        \forall K\in I_{M_\ell(x)}\, \forall \mathfrak{p}_K \in \mathscr{S}_K \, [ v_{\mathfrak{p}_K}(dx^\ell)>v_{\mathfrak{p}_K}(d^\ell)].
    \]
    
    Fix $M\in I_{M_\ell(x)}$ such that $[H_{\rm inf}: K_{\rm inf}]=[H_M:M]$. This implies that $\norm{H_{\rm inf}(\sqrt[\ell]{a})}{H_{\rm inf}}{y}=\norm{H_M(\sqrt[\ell]{a})}{H_M}{y}$ for any $y\in  H_M(\sqrt[\ell]{a})$. Thus, it suffices to find $y\in H_M(\sqrt[\ell]{a})$ such that $\norm{H_M(\sqrt[\ell]{a})}{H_M}{y}=dx^\ell+d^\ell$. By Hasse's Norm Theorem, this in turn reduces to showing that $dx^\ell+d^\ell$ is a norm locally for every prime in $H_M$.
    
    To this end, let $\mathfrak{p}_{H_M}$ be a prime over $H_M$. If $\mathfrak{p}_{H_M}\notin\mathscr{S}_{H_M}$, then by Lemma~\ref{lemma:congruentes2} we have that $\congruente{v_{\mathfrak{p}_{H_M}}(dx^p+d^p)}{0}{\ell}$ and $\congruente{v_{\mathfrak{p}_{H_M}}(a)}{0}{\ell}$. Therefore, the result follows from Lemma~\ref{lemma:rumely}. On the other hand, if $\mathfrak{p}_{H_M}\in \mathscr{S}_{H_M}$ then by hypothesis we have that $v_{\mathfrak{p}_{H_M}}(a)=0$ and $v_{\mathfrak{p}_{H_M}}(dx^\ell+d^\ell)=v_{\mathfrak{p}_{H_M}}(d^\ell)=\ell v_{\mathfrak{p}_{H_M}}(d)$. The result follows from Lemma~\ref{lemma:rumely}.
    
    Let $x\in K_{\rm inf}$ be given such that $v_{\mathfrak{p}_K}(dx^\ell)\leq v_{\mathfrak{p}_K}(d^\ell)$ for some $K\in I_{M_\ell(x)}$ and $\mathfrak{p}_K \in \mathscr{S}_K$. We need to show that $dx^\ell+d^\ell$ is not a norm. Aiming for a contradiction, suppose that there is a $y\in H_{\rm inf}(\sqrt[\ell]{a})$ such that $\norm{H_{\rm inf}(\sqrt[\ell]{a})}{H_{\rm inf}}{y}=dx^\ell+d^\ell.$ Fix $N\in I_{K}$ such that $y\in H_N(\sqrt[\ell]{a})$ and $[H_{\rm inf}(\sqrt[\ell]{a}): H_{\rm inf}]=[H_N(\sqrt[\ell]{a}): H_N]$. Hence $\norm{H_N(\sqrt[\ell]{a})}{H_N}{y}=dx^\ell+d^\ell$. Since the prime $\mathfrak{p}_{M_\ell}:=\mathfrak{p}_{H_M}\cap M_\ell$ is completely $\ell$-bounded it follows that $\ncongruente{d(\mathfrak{p}_{H_M}:\mathfrak{p}_{M_\ell})}{0}{\ell}.$ Thus, we are in the following situation:
    
    \begin{enumerate}[{\rm (i)}]
            \item $v_{\mathfrak{p}_{M}}(d)<0;$
            \item $v_{\mathfrak{p}_M}(dx^\ell)< v_{\mathfrak{p}}(d^\ell)$;
            \item the valuation $v_{\mathfrak{p}_M}(a)=0$ and $\operatorname{red}_{\mathfrak{p}_M}(a)\notin \mathbb{F}_{\mathfrak{p}_M}^{\times\ell}$; and
            \item $\ncongruente{v_{\mathfrak{p}_M}(d)}{0}{\ell}.$ 
    \end{enumerate}
    
    Fix a prime $\mathfrak{p}_N\in \mathcal{C}_N(\mathfrak{p}_M)$ such that $\ncongruente{d(\mathfrak{p}_N,\mathfrak{p}_M)}{0}{\ell}$ .

    These are the assumptions of Lemma~\ref{lemma:ncong2}, therefore $v_{\mathfrak{p}_{H_N}}(a)=0$, $a$ is not an $\ell$-th power in $H_N$, and also $ \congruente{v_{\mathfrak{p}_{H_N}}(dx^\ell+d^\ell)}{v_{\mathfrak{p}_M}(d)}{\ell}$. This contradicts clause (\ref{item:rumely_2}) of Lemma~\ref{lemma:rumely}.
\end{proof}
      
\begin{lemma}

    Let notation and hypothesis be as in Lemma \ref{mainlema2}, and define
    \[
        R_{K_{\rm inf},\mathscr{S}_{{\rm inf}}}=\{x\in K_{\rm inf}: \forall \mathfrak{q}\in \mathscr{S}_{K(x)} \, v_\mathfrak{q}(x)\geq 0\}.
    \]
    
    Then 
    \[
        R_{K_{\rm inf},\mathscr{S}_{{\rm inf}}}=\{x\in B(K_{\rm inf},\ell,a,d): \forall y\in B(K_{\rm inf},\ell,a,d) \, [xy\in B(K_{\rm inf},\ell,a,d) ] \}.
    \]
\end{lemma}

\begin{proof}
    See Lemma 3.10 in \cite{sasha}. 
\end{proof}

\begin{corollary}
    Let notation and hypothesis be as in Lemma~\ref{mainlema2} and let $w\in M_\ell$ be such that $v_\mathfrak{q}(w)=1$ for all $\mathfrak{q}\in \mathscr{S}_{M_\ell}$. Then 
    \begin{align}
        x\in \mathcal{O}^*_{K_{\rm inf},\mathscr{S}_{{\rm inf}}} \Longleftrightarrow\, & 
        \forall c\in K_{\rm inf} \left ( \frac{c-1}{w}\in R_{K_{\rm inf},\mathscr{S}_{{\rm inf}}}\right )\,
        \forall b\in K_{\rm inf}\nonumber\\& \left ( bx^\ell+b^\ell=0 \vee \exists y\in L_{\rm inf}(\sqrt[\ell]{c}) [\norm{L_{\rm inf}(\sqrt[\ell]{c})}{L_{\rm inf}}{y}=bx^\ell+b^\ell]\right ).\nonumber
    \end{align}
\end{corollary}

\begin{proof}
    In view of Lemma~\ref{mainlema:sasha}, it suffices to show that $c\in \Theta_\ell(K_{\rm inf},\mathscr{S}_{{\rm inf}}) $ if and only if $\frac{c-1}{w}\in R_{K_{\rm inf},\mathscr{S}_{{\rm inf}}}.$ On one hand, if $\frac{c-1}{w}\in R_{K_{\rm inf},\mathscr{S}_{{\rm inf}}}$ and $\mathfrak{q}\in \mathscr{S}_{K(c)} $, then $v_\mathfrak{q}\left (\frac{c-1}{w}\right )\geq 0$. Thus, $v_\mathfrak{q}(c-1)\geq v_\mathfrak{q}(w)\geq 1.$ On the other hand, let $c\in \Theta_\ell(K_{\rm inf},\mathscr{S}_{{\rm inf}}) $ and let $\mathfrak{q}\in \mathscr{S}_{M_\ell(x)}.$ Notice that $v_\mathfrak{q}(w)=e(M_\ell(x),M_\ell)v_{\mathfrak{q}_{M_\ell}}(w)=1$ as $M_\ell$ is a completely $\ell$-bounding field for $\mathfrak{q}_{M_\ell}$. Then $v_\mathfrak{q}(c-1)\geq 1=v_\mathfrak{q}(w)$, and hence $\frac{c-1}{w}\in R_{K_{\rm inf},\mathscr{S}_{{\rm inf}}}$ as required.
\end{proof}

We are now ready to prove the main Theorem of the section.

\begin{theorem}\label{thm:sasha}
    Let $p$ be a rational prime number and let $q$ be a power of $p$. Let $F$ be a finite separable geometric field extension of $\mathbb{F}_q(t)$, and let ${\mathscr S}_{F}$ be a finite nonempty set of primes of $F$. Let $\ell, \ell'$ be rational primes (not necessarily different) both coprime to the characteristic of $\mathbb{F}_q.$  Let $F_{\rm inf}$ be an infinite algebraic extension of $F$. If all primes of $F$ are hereditarily $\ell$-bounded and all primes in ${\mathscr S}_F$ are completely $\ell'$-bounded, then the integral closure $\mathcal{O}_{F_{\rm inf},\mathscr{S}_{{\rm inf}}}$ of $\mathcal{O}_{F,{\mathscr S}_F}$ in $F_{\rm inf}$ is first-order definable with parameters. More precisely, there exists a $\mathcal{L}_{ring}$ first-order formula $\varphi_{\mathscr{S}_F}(v,a,d,w)$ with free variable $v$ and parameters $a,d,w\in F_{\rm inf}$ such that 
    \[
    \forall x\in F_{\rm inf} \left [x\in \mathcal{O}_{F_{\rm inf},\mathscr{S}_{{\rm inf}}} \Leftrightarrow  \varphi_{\mathscr{S}_F}(x,a,d,w)\right ]
    \]
\end{theorem}

\begin{proof}  
    This is essentially Theorem 3.14 of \cite{sasha}. We just point out the following adjustment. If $\ell$ is a rational  prime different from the characteristic, then the norm form $N(U_1,\dots,U_\ell,C,Z)$ is a polynomial in $\mathbb{F}_p[U_1,\dots,U_q,C,Z]$ whose coefficients depend only on $\ell$ and the characteristic of the prime field.
\end{proof}

\subsection{Applications to the current work}

\begin{proposition} \label{prop:sashaapplied}
    Let $q$ be a power of $p$ and let $r$ be a prime different from $p$. Let $\mathbb{F}_q\left ( t^{r^{-\infty}}\right ):=\bigcup\limits_{n=1}^\infty \mathbb{F}_q\left ( t^{r^{-n}}\right ).$ Then the ring $\mathbb{F}_q\left [t^{r^{-\infty}}\right ]$ is first-order definable in the language $\mathcal{L}_t:=\mathcal{L}_{\rm ring}\cup\{t\}$. 
\end{proposition}

\begin{proof}
    We shall apply Theorem~\ref{thm:sasha} to the following setup.
    
    Let $F=\mathbb{F}_q(t)$, $F_{\rm inf}:=\mathbb{F}_q\left ( t^{r^{-\infty}}\right )$ and $\mathscr{S}_F=\{\infty\}$, and let $\ell$ be any rational prime different from $p$ and $r$. Let us now show that the hypotheses of Theorem~\ref{thm:sasha} hold in our setting. Let $\zeta_{r^n}$ be a primitive $r^n$-th root of unity. First notice that $\mathbb{F}_q ( t^{r^{-n}}, \zeta_{r^n} ) / \mathbb{F}_q(t )$ is cyclic extension of order $r^n$. It follows that for any prime $\mathfrak{p}$ (including the prime at infinity) of $\mathbb{F}_q( t )$ its residue field degrees and its ramification indices in $\mathbb{F}_q ( t^{r^{-n}},\zeta_{r^n} )$ divide $r^n$. By the tower law, the same holds in $\mathbb{F}_q ( t^{r^{-n}})$. Therefore, $F$ is a complete $\ell$-bounding field for all primes of $F$. Thus, the integral closure $\mathcal{O}_{F_{\rm inf},\mathscr{S}_{{\rm inf}}}$ of $\mathcal{O}_{F,{\mathscr S}_F}$ in $F_{\rm inf}$ is first-order definable with parameters. One possible choice for the parameters is $w=\frac{1}{t}, d=t$ and $a$ can be any element of $\mathbb{F}_q$ which is not an $\ell$-th power.
    
    We are left to show that $\mathbb{F}_q\left [t^{r^{-\infty}}\right ]$ is integrally closed in $F_{\rm inf}$. Since the polynomials $f_n(x,t):=x^{r^n}-t$ are irreducible in $\mathbb{F}_q^{\rm alg}[x,t]$ and the affine plane curve $Spec(\mathbb{F}_q^{\rm alg}[x,t]/\langle f_n(x,t)\rangle)$ is smooth, then its coordinate ring $\mathbb{F}_q^{\rm alg}[x,t]/\langle f_n(x,t)\rangle=\mathbb{F}^{\rm alg}_q\left [ t^{r^{-n}}\right ]$ is integrally closed in its quotient field. Thus, $\mathbb{F}_q\left [t^{r^{-\infty}}\right ]$ is the integrally closed in $F_{\rm inf}$ as required.
\end{proof}

\section{Extensions by $r^n$-th roots of the indeterminate} \label{sec:ulmer}

Throughout this section, let $p$ be a fixed odd prime. We show the following:

\begin{theorem} \label{thm:undec-roots}
    There exist infinitely many prime numbers $r$ such that, for every prime power $p^a$, the theory of the structure \[\mathbb{F}_{p^a}\left[t^{r^{-\infty}}\right]=\bigcup_{n\in\mathbb{N}}\mathbb{F}_{p^a}\left[t^{r^{-n}}\right]\]
    in the language $\{0,1,t,+,\times\}$ interprets the Theory $R$ and is thus undecidable.
\end{theorem}

In order to do, so we apply the criterion described in Theorem~\ref{thm:henson} by constructing a parametrized family of definable sets of arbitrarily large cardinalities.

\subsection{A family of elliptic curves of constant rank}

For any given positive integer $n$, let $E_n$ be the elliptic curve $y^2=x(x+1)\left(x+t^n\right)$ over $\mathbb{F}_p(t)$ or one of its extensions. We want to show the following.

\begin{theorem} \label{thm:constantrank}
    Let $r$ be a prime number such that
    \begin{itemize}
        \item $r\equiv 3\pmod{4}$
        \item $\legendre{p}{r} = 1$
    \end{itemize}
    There exist infinitely many primes $q$ such that for every $a\in\mathbb{N}^+$ there exists a positive constant $C(a)$ such that for every positive integer $n$
    \[
        \operatorname{rk}\left(E_{q}\left(\mathbb{F}_{p^a}\left(t^{r^{-n}}\right)\right)\right) = C(a).
    \]

    Moreover, the bounds $C(a)$ are themselves bounded: there exists a constant $D$ such that $C(a)\le D$ for every $a\in\mathbb{N}^+$.
\end{theorem}

Note that there are infinitely many primes $r$ that satisfy the statement of Theorem~\ref{thm:constantrank}. For the remainder of this section, let $r$ be one such prime.

Given $a,n\in\mathbb{N}$, we write
\begin{itemize}
    \item $F_{a,n} = \mathbb{F}_{p^a}\left(t^{r^{-n}}\right)$,
    \item $F_{a,\infty} = \bigcup_{n\ge 0} F_{a,n}$,
    \item $F_{\infty, n} = \bigcup_{a > 0} F_{a,n}$, and
    \item $F_{\infty, \infty} = \bigcup_{a>0} F_{a,\infty} = \bigcup_{n\ge 0} F_{\infty, n}$.
\end{itemize}

Let $q$ be an odd prime such that $\legendre{p}{q}=\legendre{q}{r}=-1$. Fix $a\in\mathbb{N}$. In order to prove Theorem~\ref{thm:constantrank}, we will rely on a theorem about the elliptic curves $E_n$ by R. Conceição, C. Hall and D. Ulmer \cite{ulmerleg2} that we will state presently. This result relates to the concept of \emph{balance of an integer modulo another integer}; we will not need the actual definition of this concept, but rather a characterisation:

\begin{lemma}[\cite{pomeranceulmer}, Theorem 2.1] \label{lemma:notbalanced}
    Let $n,m\in\mathbb{N}$. Then $n$ is not balanced modulo $m$ if and only if there exists an odd character $\chi$ modulo $m$ such that $\chi(n)=1$ and $\sum_{0<k<\frac{m}{2}}\chi(k) \neq 0$.
\end{lemma}

\begin{theorem}[\cite{ulmerleg2}, Theorem 2.2] \label{thm:ulmer}
    Let $p_0$ be an odd prime, and let $d$ be a positive integer not divisible by $p$. The rank of the group $E_1\left(\mathbb{F}_{p_0^n}\left(t^{d^{-1}}\right)\right)$ is equal to
    \[
        \sum_{\substack{e\mid d\\ e>2\\ p_0\textrm{ balanced mod }e}} \left|\left({\mathbb{Z}}/{e\mathbb{Z}}\right)^{\times}\,:\, \left< p_0^n\right> \right|
    \]
    where $\left< p_0^n\right>$ is the subgroup of powers of $p_0^n$ inside each $\left({\mathbb{Z}}/{e\mathbb{Z}}\right)^{\times}$.
\end{theorem}

In our setting, we will use this theorem in the following way:

\begin{corollary} \label{coro:fromulmer}
    Write $d=qr^n$ in the previous theorem. The rank of the group $E_q\left(F_{a,n}\right)$ is equal to
    \[
        \sum_{\substack{e\mid qr^n\\ p\textrm{ balanced mod }e}} \left|\left({\mathbb{Z}}/{e\mathbb{Z}}\right)^{\times}\,:\, \left< p^a\right> \right|.
    \]
\end{corollary}

In order to prove Theorem~\ref{thm:constantrank}, we need to show that the number of divisors $e$ of $qr^n$ such that $p$ is balanced modulo $e$ does not increase with $n$. We will require several lemmas.

\begin{lemma} \label{lemma:balanced_modulus}
    Let $x,y$ be distinct odd primes. If $\legendre{x}{y} = -1$, then $x$ is balanced modulo $y$.
\end{lemma}

\begin{proof}
    As the group of units modulo $y$ is cyclic of even order and $x$ is not a square in it then the order of $x$ modulo $y$ is an even integer $2k$. Thus $x^k$ modulo $y$ is a square root of $1$ distinct from $1$; as the integers modulo $y$ form a field this means that $x^k=-1$.

    Let $\chi$ be any character modulo $y$ such that $\chi(x)=1$. Then $\chi(-1)=1^k=1$, hence $\chi$ is not odd. Using Lemma~\ref{lemma:notbalanced} we conclude that $x$ is balanced modulo $y$.
\end{proof}

\begin{lemma} \label{lemma:notbalanced_modulus}
    Let $x,y$ be distinct odd primes such that $y\equiv 3\pmod{4}$ and $\legendre{x}{y} = 1$. Then $x$ is not balanced modulo $y$.
\end{lemma}

\begin{proof}
    By the first supplement to quadratic reciprocity, $\legendre{-1}{y} = -1$. Consider the character $\chi$ modulo $y$ given by $\chi(a) = \legendre{a}{y}$. It is odd and such that $\chi(x)=1$. Moreover, as the sum
    \[
        \sum_{0<k<\frac{y}{2}}\chi(k)
    \]
    has an odd number of terms, then it is nonzero. By Lemma~\ref{lemma:notbalanced}, $x$ is not balanced modulo $y$.
\end{proof}

\begin{lemma} \label{lemma:going_up}
    Let $x,y$ be distinct odd primes. Let $z$ be odd, divisible by $y$, and coprime to $x$. If $x$ is not balanced modulo $z$, then it is not balanced modulo $yz$ either.
\end{lemma}

\begin{proof}
    Let $\chi_0$ be an odd character modulo $z$ witnessing the properties stated in Lemma~\ref{lemma:notbalanced}. We define a new character $\chi$ modulo $yz$ as $\chi(k) = \chi_0(k)$. This is well-defined as $y$ is a prime factor of $z$ already.

    This character is odd and satisfies $\chi(p)=\chi_0(p)=1$. It remains to show that $\sum_{0<k<\frac{yz}{2}}\chi(k)\neq 0$. We have
    \[
        \sum_{0<k<\frac{yz}{2}}\chi(k) = \sum_{i=1}^{\lfloor\frac{y}{2}\rfloor}\left(\sum_{(i-1)z<k<iz}\chi(k)\right) + \sum_{\lfloor\frac{y}{2}\rfloor z<k<\frac{yz}{2}}\chi(k)
    \]
    \[
        \sum_{0<k<\frac{yz}{2}}\chi(k) = \sum_{i=1}^{\lfloor\frac{y}{2}\rfloor}\left(\sum_{0<k<z}\chi_0(k)\right) + \sum_{0<k<\frac{z}{2}} \chi_0(k) = 0 + \sum_{0<k<\frac{z}{2}} \neq 0. 
    \]
\end{proof}

\begin{proof}[Proof of Theorem~\ref{thm:constantrank}]
    By choice of $q$ and Lemma~\ref{lemma:balanced_modulus}, $p$ is balanced modulo $q$. By choice of $r$ and Lemma~\ref{lemma:notbalanced_modulus}, $p$ is not balanced modulo $r$; moreover, by Lemma~\ref{lemma:going_up} $p$ is not balanced modulo $r^i$ for any exponent $i$.

    We aim to show that $q$ is the only divisor of $qr^n$ such that $p$ is balanced modulo it; the conclusion would follow from Corollary~\ref{coro:fromulmer}. In order to do this it suffices to show that $p$ is not balanced modulo $qr$; the cases of all other divisors of the form $qr^i$ would then follow from Lemma~\ref{lemma:going_up}.

    Consider the characters $\chi(k)=\legendre{k}{r}$ modulo $qr$ and $\chi_0(k)=\legendre{k}{r}$ modulo $r$. From the proof of Lemma~\ref{lemma:notbalanced_modulus} we know that $\sum_{0<k<\frac{r}{2}}\chi_0(k)\neq 0$. As $q$ and $r$ are distinct primes,
    \[
        \chi(k) = \begin{cases}
            \chi_0(k) & \textrm{ if }q\nmid k \\
            0 & \textrm{ if }q\mid k
        \end{cases}.
    \]

    Hence we obtain
    \[
        \sum_{0<k<\frac{qr}{2}}\chi(k) = \sum_{0<k<\frac{qr}{2}}\chi_0(k) - \sum_{0<k<\frac{r}{2}} \chi_0(qk).
    \]

    As $\sum_{(i-1)r<k<ir}\chi_0(k) = 0$, the previous equation reduces to
    \[
        \sum_{0<k<\frac{qr}{2}}\chi(k) = \sum_{0<k<\frac{r}{2}} \chi_0(k) - \chi_0(q)\sum_{0<k<\frac{r}{2}} \chi_0(k) = 2\sum_{0<k<\frac{r}{2}} \chi_0(k)\neq 0
    \]
    and as $\chi$ is odd and $\chi(p)=1$ we conclude that $p$ is not balanced modulo $qr$.

    Hence
    \[
        \operatorname{rk}\left(E_q(F_{a,n})\right) = \sum_{\substack{e\mid qr^n\\ p\textrm{ balanced mod }e}} \left|\left({\mathbb{Z}}/{e\mathbb{Z}}\right)^{\times}\,:\, \left< p^a\right> \right| = \left|\left({\mathbb{Z}}/{q\mathbb{Z}}\right)^{\times}\,:\, \left< p^a\right> \right|
    \]
    which does not depend on $n$. Moreover, this index is bounded above by $\left|\left({\mathbb{Z}}/{q\mathbb{Z}}\right)^{\times}\right|$, which does not depend on $a$ either.
\end{proof}

\subsection{An elliptic curve on the compositum of fields}

The aim of this section is to prove the following.

\begin{theorem} \label{thm:curvecompositum}
    Fix $a\in\mathbb{N}^+\cup\{\infty\}$. There exists a positive integer $n_0$ such that
    \[
        E_q\left(F_{a,\infty}\right) = E_q\left(F_{a,n_0}\right).
    \]
\end{theorem}

We will require some results:

\begin{lemma}[\cite{ulmerleg1}, Proposition 6.1] \label{lemma:finitetorsion}
    For any $a,n\in\mathbb{N}\cup\{\infty\}$, the torsion of $E_q\left(F_{a,n}\right)$ is isomorphic to ${\mathbb{Z}}/{2\mathbb{Z}}\times{\mathbb{Z}}/{2\mathbb{Z}}$.
\end{lemma}

A quick remark on Lemma~\ref{lemma:finitetorsion}: Ulmer's original statement covers only the case where $n$ is finite. The additional case for the compositum follows from the facts that it is a compositum of nested fields -- and that the curve $E_q$ is defined by the same equation in all of them.

\begin{lemma} \label{lemma:torsionquotient}
    For any $a\in\mathbb{N}^+$, the quotient group ${E_q\left(F_{a,\infty}\right)}/{E_q\left(F_{a,0}\right)}$ is torsion.
\end{lemma}

\begin{proof}
    Let $P\in E_q\left(F_{a,\infty}\right)$. There exists a positive integer $n$ such that $P\in E_q\left(F_{a,n}\right)$. As $E_q\left(F_{a,0}\right) \subseteq E_q\left(F_{a,n}\right)$ and they both have the same rank by Theorem~\ref{thm:constantrank}, there exists an integer $m$ such that $mP\in E_q\left(F_{a,0}\right)$.
\end{proof}

\begin{proof}[Proof of Theorem~\ref{thm:curvecompositum} for finite exponent $a$]
    We follow the ideas of D. Rohrlich in \cite{rohrlich1984}, \S 3. Fix $P\in E_q\left(F_{a,\infty}\right)$. Let $n$ be the smallest integer such that $P\in E_q\left(F_{a,n}\right)$. We aim to show that there exists an upper bound $n_0$ for $n$ that does not depend on $P$. We only need to consider the case $n\neq 0$.
    
    Let $\bar{a}$ be a multiple of $a$ such that the polynomial $x^{r^{n}}-1$ splits in $\mathbb{F}_{p^{\bar{a}}}$. Hence the extension of fields $F_{\bar{a},n} / F_{\bar{a},0}$ is Galois.

    Then $P\in E_q\left(F_{a,n}\right)\subseteq E_q\left(F_{\bar{a},n}\right)$, and as $F_{a,0} = F_{a,n}\cap F_{\bar{a},0}$ we conclude that $P\not\in E_q\left(F_{\bar{a},0}\right)$. Let $\sigma\in\operatorname{Gal}\left(F_{\bar{a},n} / F_{\bar{a},0}\right)$ be arbitrary.

    Let $m$ be the smallest integer such that $mP\in E_q\left(F_{\bar{a},0}\right)$ (cf. Lemma~\ref{lemma:torsionquotient}). Then $m\left(\sigma(P)-P\right)=0$, and by Lemma~\ref{lemma:finitetorsion} we can conclude that $m=2$. Thus $\sigma(2P)=2P$ and, as the extension is Galois, $2P\in E_q\left(F_{\bar{a},0}\right)$.

    Hence $2P\in E_q\left(F_{a,n}\right) \cap E_q\left(F_{\bar{a},0}\right) = E_q\left(F_{a,0}\right)$. As the coordinate components of $2P$ lie in $F_{a,0}$ and they are obtained from the components of $P$ via fixed rational functions (ie. not depending on $P$), we conclude that there is a bound on the degree of the field generated over $F_{a,0}$ by the components of $P$. As the degrees of $F_{a,i} / F_{a,0}$ are strictly increasing as $i$ increases, we arrive at the existence of the aforementioned upper bound $n_0$.
\end{proof}

It remains to prove the case where $a=\infty$. We require similar results before doing so.

\begin{lemma} \label{lemma:Ftilderankconstant}
    There exists $A\in\mathbb{N}^+$ such that, for $a\ge A$, the rank of $E_q\left(F_{a,n}\right)$ depends only on $n$. Furthermore, there exists $N\in\mathbb{N}$ such that for $a\ge A$, $n\ge N$, the rank of $E_q\left(F_{a,n}\right)$ is constant.
\end{lemma}

\begin{proof}
    These assertions follow from the facts that $\operatorname{rk}\left(E_q\left(F_{a,n}\right)\right)$ is monotone increasing on both $a$; and that there are only finitely many possible values for these ranks (cf. the final assertion of Theorem~\ref{thm:constantrank}).
\end{proof}

The following corollary follows the same proof as Lemma~\ref{lemma:torsionquotient}.

\begin{corollary} \label{coro:Ftildetorsionquotient}
    For $A$ as in the previous lemma and any $n\in\mathbb{N}$, the quotient group ${E_q\left(F_{\infty,n}\right)}/{E_q\left(F_{A,n}\right)}$ is torsion.
\end{corollary}

\begin{lemma} \label{lemma:verticalrohrlich}
    Fix $n\in\mathbb{N}$. There exists a positive integer $a_0$ such that \[E_q \left(F_{\infty,n}\right) = E_q\left(F_{a_0,n}\right).\] 
\end{lemma}

\begin{proof}
    We replicate the proof of Theorem~\ref{thm:curvecompositum} for finite exponent $a$, swapping the roles of the subindices of $F_{a,n}$. The extensions $F_{a,n} / F_{1,n}$ are already Galois, thus we don't need to move to a larger field; and Corollary~\ref{coro:Ftildetorsionquotient} fills the role of Lemma~\ref{lemma:torsionquotient}.
\end{proof}

\begin{proof}[Proof of Theorem~\ref{thm:curvecompositum} for $a=\infty$]
    By Lemmas~\ref{lemma:Ftilderankconstant} and \ref{coro:Ftildetorsionquotient}, there exists a positive integer $N_0$ such that for $n\ge N_0$ the rank of $E_q\left(F_{\infty,n}\right)$ is constant. As $F_{\infty,\infty} = \bigcup_{n\ge N_0}F_{\infty,n}$ we can replicate the proof of Lemma~\ref{lemma:torsionquotient} to show that the quotient group ${E_q\left(F_{\infty,\infty}\right)}/{E_q\left(F_{\infty,N_0}\right)}$ is torsion.

    We can now follow the same proof as for the finite case owing to the fact that the extensions $F_{\infty,n} / F_{\infty,0}$ are Galois.
\end{proof}

\subsection{Families with increasing finite sets}

The aim of this section is to show the following.

\begin{theorem} \label{thm:powers-henson-like}
    Let $p,r$ be as in the statement of Theorem~\ref{thm:constantrank}, and let $a\in\mathbb{N}^+\cup\{\infty\}$. There exists a parametrised family of definable sets in $\mathbb{F}_{p^a}\left[t^{r^{-\infty}}\right]$ that consists of sets of arbitrarily large finite cardinalities.
\end{theorem}

We will require two theorems from the bibliography:

\begin{theorem}[\cite{ulmerleg1}, \S 2]
    The $j$-invariant of $E_1$ is given by $\displaystyle\frac{256\left(t^2-t+1\right)^3}{t^2\left(t-1\right)^2}$. In particular, it is nonconstant.
\end{theorem}

\begin{theorem}[\cite{voloch1990}, Theorem 5.3] \label{thm:voloch}
    Let $K$ be a function field in one variable over a finite field and $S$ a finite set of places of $K$. If $E / K$ is an elliptic curve with nonconstant $j$-invariant and $f\in K(E)$ is a nonconstant function then there are only finitely many points $P\in E(K)$ such that $f(P)$ is $S$-integral.
\end{theorem}

\begin{proof}[Proof of Theorem~\ref{thm:powers-henson-like}]
    First consider the case where $a$ is finite. Let $q$ be as in Theorem~\ref{thm:constantrank}. Write $R = \mathbb{F}_{p^a}\left[t^{r^{-\infty}}\right]$, $K= \mathbb{F}_{p^a}\left(t^{r^{-\infty}}\right)$. By Theorem~\ref{thm:curvecompositum} there exists $n$ such that $E_q(K) = E_q\left(\mathbb{F}_{p^a}\left(t^{r^{-n}}\right) \right)$; write $K' =  \left(\mathbb{F}_{p^a}\left(t^{r^{-n}}\right) \right)$ and $R' = K'\cap R$.

    Let $\varphi(x,y,z) \,:\, y^2z = x(x+z)\left(x+zt^q\right)$. Using this formula we define a family of definable sets in $R$\,:
    \[
        C_{\lambda} = \left\{x\in R\,:\, \exists y\,\left(R\models \varphi(x,y,\lambda)\right) \right\}.
    \]
    We claim that these sets $C_{\lambda}$, as $\lambda$ varies over $R$, are of arbitrarily large finite cardinalities.

    \emph{Each $C_{\lambda}$ is finite:} we first notice that $C_{\lambda}$ is finite for $\lambda = 0$. Fix $\lambda\neq 0$, and let $x_0\in C_{\lambda}$. Let $y_0$ be such that $K\models\varphi(x_0,y_0,\lambda)$, and $P_0=\left(\lambda^{-1}x_0,\lambda^{-1}y_0\right)\in K^2$; by definition of $\varphi$ we see that $P_0\in E_q(K)=E_q(K')$.

    Let $\lambda_0 \in R'$ be a nonzero multiple of $\lambda$ in $R$; let $f\in K'(E_q)$ be given by $f(P) = \lambda_0 x(P)$. Then $f(P_0) = \frac{\lambda_0}{\lambda} x_0\in R$. By Theorem~\ref{thm:voloch}, there are finitely many such points $P_0$; hence there are finitely many possible values of $x_0\in C_{\lambda}$.

    \emph{The sets $C_{\lambda}$ are arbitrarily large:} let $N\in\mathbb{N}$. As the rank of $E_q(K)$ is positive, there exist $N$ distinct points $P_1,\ldots , P_N\in E_q(K) \cap K^2$. For each such point there exists $\lambda_i \in R$ such that $P_i \in \left(\frac{1}{\lambda_i}R\right)^2$. Let $\lambda$ be the product of the $\lambda_i$. Then, for every $i$, $R\models \varphi\left(\lambda x(P_i), \lambda y(P_i), \lambda\right)$, hence $C_{\lambda}$ contains all the $\lambda x(P_i)$ and has cardinality of at least $N$.

    This concludes the proof for $a$ finite. The additional case for $a=\infty$ follows from Theorem~\ref{thm:curvecompositum}, Lemma~\ref{lemma:verticalrohrlich} and the fact that the definable family $C_\lambda$ does not depend on $a$.
\end{proof}

\section{Undecidability of infinite extensions in the language of rings} \label{sec:mainthm}

Throughout the previous sections we have been setting up a proof of the undecidability of the field $\mathbb{F}_{p^a}\left(t^{r^{-\infty}}\right)$ by first showing that the ring $\mathbb{F}_{p^a}\left[t^{r^{-\infty}}\right]$ is definable and then finding a uniformly parametrizable family of sets in this ring that satisfies the conditions of the criterion given in Theorem~\ref{thm:henson}. All of this in the language of rings extended with a symbol for the indeterminate $t$.

In this section we will show that this fact does not require this extension of the language:

\begin{theorem} \label{thm:main}
    Let $p$ be an odd prime, $a$ a positive integer and $r$ a prime congruent to $3$ $\pmod{4}$ and such that $p$ is a quadratic residue modulo $r$.

    The first-order theory of the structure $\mathbb{F}_{p^a}\left(t^{r^{-\infty}}\right)$ in the language of rings $\{0,1,+,\times\}$ interprets the theory R of Tarski, Mostowski and Robinson. In particular, it is undecidable.

\end{theorem}

\begin{remark}
    For any given $p$, there are infinitely many such $r$. More specifically, by the prime number theorem on arithmetic progressions the Dirichlet density of the set of such $r$ is $\frac{1}{4}$.
\end{remark}

We will require a technical lemma:

\begin{lemma} \label{lemma:polynomialinpowers}
    Let $F$ be a field and $n\ge 1$ be an integer. For every polynomial $p\in F[t]$ there exists a nonzero $\bar{p}\in F[t]$ such that $p\bar{p}\in F\left[t^n\right]$.
\end{lemma}

\begin{proof}
    Let $a_1,\dots ,a_k$ be the roots (with repetitions) of $p$ in a splitting field $\overline{F}$. For each $i$, let $q_i(t)\in F[t]$ be the minimal polynomial of $a_i^n$ over the extension $\overline{F}/F$. Then $\prod_i q_i(t^n)$ is a multiple of $p$.
\end{proof}

\begin{proof}[Proof of Theorem~\ref{thm:main}]
    Write $\mathcal{K}=\mathbb{F}_{p^a}\left(t^{r^{-\infty}}\right)$,  $\mathcal{O}=\mathbb{F}_{p^a}\left[t^{r^{-\infty}}\right]$ and $\mathcal{O}_-=\mathbb{F}_{p^a}\left[t^{-r^{-\infty}}\right]$.

    By Theorem~\ref{thm:sasha}, there exists a first-order formula with $t$ as a parameter that defines $\mathcal{O}$ in $\mathcal{K}$. Substituting a variable $y$ at every instance of $t$ in this formula, we obtain a formula $\phi(x,y)$ in the language of rings such that $\mathcal{O}=\{x\,:\,\mathcal{K}\models \phi(x,t)\}$.

    Write $R_y=\{x\,:\,\mathcal{K}\models \phi(x,y)\}$. We define the set $\mathcal{T}$ in $\mathcal{K}$ as the set of all $y\in \mathcal{K}$ that satisfy the following conditions:
    \begin{enumerate}[(i)]
        \item \label{condi:nonzero}$y\neq 0$;
        \item \label{condi:ring}$R_y$ is a ring;
        \item \label{condi:overFq}$R_y$ contains every solution of $x^{p^a}-x=0$ in $\mathcal{K}$;
        \item \label{condi:noextraunits}every unit of $R_y$ is a solution of $x^{p^a}-x=0$;
        \item \label{condi:containsy}$y\in R_y$;
        \item \label{condi:ynonunit}$y$ is not a unit of $R_y$;
        \item \label{condi:yhasrtower}for every $z\in R_y$ that divides $y$ in $R_y$, there exists a unit $u$ of $R_y$ such that $uz$ is a $r$-th power in $R_y$.
    \end{enumerate}

    It is clear that $\mathcal{T}$ is definable in the corresponding languages and that $t\in\mathcal{T}$. Moreover, as the map $t\mapsto t^{-1}$ is an isomorphism between $\left(\mathcal{K};0,1,+,\times, t\right)$ and $\left(\mathcal{K};0,1,+,\times, t^{-1}\right)$ that maps $\mathcal{O}$ to $\mathcal{O}_-$, then $t^{-1}\in\mathcal{T}$ with $R_{t^{-1}}=\mathcal{O}_-$.

    Given $y\in\mathcal{T}$, we note the following: by condition (\ref{condi:overFq}), $R_y$ contains the constants in $\mathcal{K}$; by conditions (\ref{condi:nonzero}), (\ref{condi:noextraunits}) and (\ref{condi:ynonunit}), $y$ is not a constant, and from conditions (\ref{condi:ring}) and (\ref{condi:containsy}) we get that $\mathbb{F}_{p^a}[y]\subseteq R_y$. Finally, from condition (\ref{condi:yhasrtower}) we conclude that $\mathbb{F}_{p^a}[y^{r^{-n}}]\subseteq R_y$ for every $n$, hence $y$ must be associated to a power of $t$.

    Thus $\mathcal{T}\subseteq \left\{ut^n\,:\,u\in \mathbb{F}_{p^a}^\times\,,\,n\in\mathbb{Z}\left[\frac{1}{r}\right]\setminus\{0\}\right\}$.

    We aim to distinguish $\mathcal{O}$ from within the family $\{R_y\}_{y\in\mathcal{T}}$. In order to do so, we first define the subset $\mathcal{T}'$ of $\mathcal{T}$ as the set of all $y\in\mathcal{T}$ that satisfy the following condition:
    \begin{itemize}
        \item for every $z\in\mathcal{T}$ such that $z\in R_y$ and every $w\in R_y$ that divides $z$ in $R_y$ there exists a unit $u$ of $R_y$ such that $uw$ is a $r$-th power in $R_y$.
    \end{itemize}

    Clearly $\{t,t^{-1}\}\subset\mathcal{T}'$. We now specialize further by defining a subset $\mathcal{T}''$ of $\mathcal{T}'$ as the set of all $y\in\mathcal{T}'$ that satisfy the following condition:

    \begin{itemize}
        \item for every $z\in\mathcal{T}'$ such that $y\in R_z$ and for every $w_1,w_2\in R_y\cap R_z$, if $w_1$ divides $w_2$ in $R_y$ then this divisibility also holds in $R_z$.
    \end{itemize}

    We note that $t\in\mathcal{T}''$: indeed, if $z\in\mathcal{T}'$ is such that $t\in R_z$, then every positive power of $t$ belongs to $R_z$ by the definition of $\mathcal{T}'$; and if $w_1,w_2\in R_t=\mathcal{O}$ are such that $w_1$ divides $w_2$ in $\mathcal{O}$, then $\frac{w_2}{w_1}\in\mathcal{O}$ is a linear combination of positive powers of $t$, hence $\frac{w_2}{w_1}$ in $R_z$.

    We note further that for every $y\in \mathcal{T}''$ either $R_y\in\mathcal{O}$ or $R_y\in\mathcal{O}_-$. Fix $y\in\mathcal{T}''$. Without loss of generality, assume that $y$ is associate to a positive power of $t$, hence $y\in R_t=\mathcal{O}$. Suppose there exists $\frac{w_2}{w_1}\in R_y\setminus\mathcal{O}$, with $w_1,w_2\in\mathcal{O}$. As every map of the form $t\mapsto t^{r^n}$ with $n\in\mathbb{Z}$ is an automorphism of $\mathcal{K}$, we may assume that $w_1,w_2$ are polynomials in $t$ and $y$ is associate to a positive integer power of $t$, $y=ut^n$.

    Thus $\mathbb{F}_{p^a}[y]=\mathbb{F}_{p^a}[t^n]\subseteq R_y$; by Lemma~\ref{lemma:polynomialinpowers} there exists $\overline{w_1}\in\mathbb{F}_{p^a}[t]$ such that $w_1\overline{w_1}\in\mathbb{F}_{p^a}[t^n]\subseteq R_y$. As $w_2\overline{w_1}=\frac{w_2}{w_1}w_1\overline{w_1}\in R_y$, then both $w_1\overline{w_1}$ and $w_2\overline{w_1}$ are in $R_y\cap R_t$. As $\frac{w_2}{w_1}=\frac{w_2\overline{w_1}}{w_1\overline{w_1}}\in R_y$, by definition of $\mathcal{T}''$ we conclude that $\frac{w_2}{w_1}\in R_t$ and thus $R_y\subseteq R_t=\mathcal{O}$.

    Similarly, if $y$ is associated to a negative power of $t$ then $R_y\subseteq \mathcal{O}_-$.

    Thus $\mathcal{O}$ and $\mathcal{O}_-$ are the only two maximal elements of $\{R_y\}_{y\in\mathcal{T}''}$. We define $\mathcal{T}_0$ as the set of all $y\in\mathcal{T}''$ such that
    \begin{itemize}
        \item for every $z\in\mathcal{T}''$, if $R_z\subseteq R_y$ then $R_z=R_y$.
    \end{itemize}

    And we have $\{R_y\}_{y\in\mathcal{T}_0}=\{\mathcal{O},\mathcal{O}_-\}$. Note that $\mathcal{O}$ and $\mathcal{O}_-$ are isomorphic as rings via the map $t\mapsto t^{-1}$, so this refinement of $\mathcal{T}$ cannot go any further.
    
    As $\mathcal{O}$ and $\mathcal{O}_-$ are isomorphic as rings, any formula $F$ in the language of rings with variables $x_1,\dots x_n$ we wish to interpret over $\mathcal{O}$ can be equivalently stated over $\mathcal{K}$ by \[\exists y\in\mathcal{T}_0\,\left(x_1\in R_y\,\wedge\,\ldots \,\wedge\,x_n\in R_y\,\wedge\,F\right).\]

    Finally, by Theorem~\ref{thm:powers-henson-like} there exists a parametrised family of definable sets of $\mathcal{O}$, defined in the language $\{0,1,+,\times ,t\}$, that consists of sets of arbitrarily large finite cardinalities. If we pick a new variable and substitute it into the definition of this family at every instance of $t$, we obtain a definable family of $\mathcal{O}$ in the language of rings that contains sets of arbitrarily large cardinalities; by Corollary~\ref{coro:hensonapplied} we conclude that the first-order theory of $\mathcal{O}$ in the language of rings is undecidable; hence this is also the case for the first-order theory of $\mathcal{K}$ in the languages described in this theorem.
\end{proof}

\begin{remark}
    Following on Lemma~\ref{lemma:verticalrohrlich} and the proof of Proposition~\ref{prop:sashaapplied}, the proof of Theorem~\ref{thm:main} can be extended to structures of the form $K\left(t^{r^{-\infty}}\right)$, where $K$ is an infinite algebraic extension of $\mathbb{F}_p$ with the following restriction: there exists an integer $l$ coprime to the characteristic such that $K$ contains a non $l$-th power.

    In this case, the language has to be extended to $\{0,1,+,\times ,K\}$.
\end{remark}

\begin{remark}
    The interpretation of the theory R in Theorem~\ref{thm:main} can be done uniformly over all finite extensions of $\mathbb{F}_p$ by working in the language $\{0,1,+,\times, C\}$, where $C$ is interpreted as the unary predicate ``is a constant''.

    This requires removing the power from the definition of the elliptic curve in the proof of Theorem~\ref{thm:powers-henson-like}, but this is fixed by the steps taken to define $\mathcal{T}'$ in the proof of Theorem~\ref{thm:main}: we will be able to conclude that $\{t^q,t^{-q}\}\in\mathcal{T}'$.
\end{remark}

\addcontentsline{toc}{section}{References}

\printbibliography
\noindent Carlos A. Mart\'inez-Ranero\\
Email: cmartinezr@udec.cl\\
Homepage: www2.udec.cl/~cmartinezr\\
\noindent Dubraska Salcedo\\
Email: dsalcedo@udec.cl\\
\noindent Javier Utreras\\
Email: javierutreras@udec.cl \hspace{10pt} javutreras@gmail.com\\

\noindent Same address: \\
Universidad de Concepci\'on, Concepci\'on, Chile\\
Facultad de Ciencias F\'isicas y Matem\'aticas\\
Departamento de Matem\'atica\\
\end{document}